\documentclass[11pt,oneside,english,reqno]{amsart}
\usepackage[T1]{fontenc}
\usepackage[utf8]{inputenc}
\usepackage{color}
\usepackage{mathrsfs}
\usepackage{amstext}
\usepackage{amsthm}
\usepackage{amssymb}
\usepackage{graphicx}
\usepackage[unicode=true,pdfusetitle,
 bookmarks=true,bookmarksnumbered=false,bookmarksopen=false,
 breaklinks=false,pdfborder={0 0 0},backref=false,colorlinks=true]
 {hyperref}
\hypersetup{
 citecolor=blue}

\makeatletter
\numberwithin{equation}{section}
\numberwithin{figure}{section}
\theoremstyle{plain}
\newtheorem*{cor*}{\protect\corollaryname}
\theoremstyle{plain}
\newtheorem*{note}{Note}
\theoremstyle{plain}
\newtheorem{thm}{\protect\theoremname}[section]
\theoremstyle{definition}
\newtheorem{defn}[thm]{\protect\definitionname}
\theoremstyle{remark}
\newtheorem{rem}[thm]{\protect\remarkname}
\theoremstyle{plain}
\newtheorem{prop}[thm]{\protect\propositionname}
\theoremstyle{plain}
\newtheorem{lem}[thm]{\protect\lemmaname}
\theoremstyle{plain}
\newtheorem{cor}[thm]{\protect\corollaryname}
\theoremstyle{plain}

\usepackage[bottom=1in]{geometry}
\usepackage{amsmath}
\usepackage{amsthm}

\numberwithin{equation}{section}
\numberwithin{figure}{section}
\usepackage{enumitem}		
 \let\footnote=\endnote
\@ifundefined{lettrine}{\usepackage{lettrine}}{}

\theoremstyle{definition}
\newtheorem{thmx}{Theorem}

\def\a{\alpha}
\def\A{\mathcal{A}}

\def\s{\sigma}

\def\L{\Lambda}

\def\R{\mathbb{R}}

\def\L{\mathcal{L}}

\def\N{\mathbb{N}}
\def\Z{\mathbb{Z}}

\def\ep{\varepsilon}

\def\I{\mathsf{I}}
\def\J{\mathsf{J}}
\def\K{\mathsf{K}}

\def\V{\mathbb{V}}
\def\W{\mathcal{W}}

\def\ol{\overline}
\def\id{\text{id}}

\newcommand{\Wloc}{\mathcal{W}_{\text{loc}}}

\def\vp{\varphi}
\def\hol{H\"older }

\def\Sig{\Sigma_T}
\def\M{\mathcal{M}}

\def\glr{\text{GL}_d(\R)}

\newcommand{\eps}{\epsilon}
\renewcommand{\phi}{\varphi}

\newcommand{\excise}[1]{}
\providecommand{\x}{\times}

\usepackage{float}

\address{Department of Mathematics, Statistics, and Computer Science,
University of Illinois Chicago, Chicago, IL, USA} 
\email{bcall@uic.edu}

\address{School of Mathematics, KIAS, 85 Hoegiro, Dongdaemun-gu, Seoul, 02455, Korea} 
\email{kiho.park12@gmail.com}

\makeatother

  \providecommand{\corollaryname}{Corollary}
  \providecommand{\definitionname}{Definition}
  \providecommand{\lemmaname}{Lemma}
  \providecommand{\propositionname}{Proposition}
  \providecommand{\remarkname}{Remark}
  \providecommand{\theoremname}{Theorem}
\providecommand{\theoremname}{Theorem}

\usepackage{ulem}

\newcommand{\comment}[1]{{}}
\newcommand{\abs}[1]{\left|#1\right|}

\begin{document}
\title[Bernoulli property of subadditive equilibrium states]{Bernoulli property of subadditive equilibrium states}
\author{Benjamin Call, Kiho Park}
\date{\today}


\begin{abstract}
Under mild assumptions, we show that the unique subadditive equilibrium states for fiber-bunched cocycles are Bernoulli. We achieve this by showing these equilibrium states are absolutely continuous with respect to a product measure, and then using the Kolmogorov property of these measures.
\end{abstract}

\maketitle

\section{Introduction}
In this paper, we study the Bernoulli property on a class of subadditive equilibrium states. We say a dynamical system $(X,f,\mathcal{B},\mu)$ is \textit{Bernoulli} if it is measurably isomorphic to a Bernoulli shift; see Definition \ref{defn: B}. Bernoullicity is the strongest ergodic property associated to complete randomness.

	The study of Bernoulli systems was vastly enriched by the celebrated work of Ornstein \cite{ornstein1970isomorphic, ornstein1973KnotBernoulli, ornstein1973geodesic}. Among other things, he showed that Bernoulli shifts are classified by their entropy up to measurable isomorphism. He also helped introduce many new definitions that imply the Bernoulli property; including the Weak Bernoulli, Very Weak Bernoulli, and finitely determined properties \cite{FriedmanOrnstein,OrnsteinWeissFinDet,Ornstein1970}. Some of these, such as the Very Weak Bernoulli property, are easier to verify than others, and we will make use of it in establishing the Bernoullicity of the subadditive equilibrium states.

Thanks to these checkable conditions for Bernoullicity, Ornstein theory has successfully been implemented to show that many natural invariant measures of important dynamical systems are Bernoulli. Ornstein and Weiss \cite{ornstein1973geodesic} showed that geodesic flows on negatively curved surfaces are Bernoulli with respect to Liouville measure. This was extended by Pesin \cite{pesin1977characteristic} for geodesic flows on a more general class of manifolds. More recently, the first-named author and Thompson \cite{call2019equilibrium} have shown that the measure of maximal entropy for the geodesic flow on non-positively curved rank one manifolds is Bernoulli, and Carrasco and Rodriguez-Hertz \cite{carrasco2021equilibrium} have recently shown that some equilibrium states for center isometries are Bernoulli as well. All of these works proceed by first showing that a system is $K$, and then showing that it is Bernoulli. General arguments can be found to show that smooth $K$-systems with some amount of hyperbolic structure \cite{chernov1996nonuniformly, ornstein1998bernoulli} are Bernoulli.
The main content of this paper is of a similar flavor.

For the first main result of the paper, we will now focus on the case where the base dynamical system is a two-sided mixing subshift of finite type $(\Sig,\s)$.
In this setting, Bowen \cite{bowen1974some} showed uniqueness of the equilibrium states for \hol potentials. 
These are important invariant measures for the underlying dynamical systems that often capture relevant statistical information, and they are also known to be Bernoulli \cite{bowen1974bernoulli, bowen1975ergodic}.

Recently, there have been great developments in the theory of subadditive thermodynamic formalism and its applications; see \cite{barreira2011dimension} and \cite{cao2019dimension}.
In particular, due to their connections to the dimension theory of fractals, subadditive potentials arising from matrix cocycles and their equilibrium states have been extensively studied; see \cite{MR3820437} and references therein. The norm potentials and singular value potentials are prime examples of such subadditive potentials (see Subsection \ref{subsec: cocycles}).
While the equilibrium states for the norm potentials are not necessarily unique in general, there are mild conditions on cocycles that guarantee the uniqueness of the equilibrium states. 

In the case of locally constant cocycles, Feng \cite{feng2003lyapunov, feng2009lyapunov} showed that irreducibility implies uniqueness of equilibrium states. 
Morris \cite{morris2018ergodic,morris2019necessary} then showed that total ergodicity of these equilibrium states implies mixing and that the failure of mixing can be characterized by certain structures of the cocycle. He recently improved his own result by showing that total ergodicity implies Bernoullicity \cite{morris2020totally}.
Under a different set of assumptions on the cocycle, Piraino \cite{piraino2018weak} showed that the equilibrium states are Bernoulli. While both results amount to establishing the Weak Bernoulli property, their methods of proof differ both from each other, and from the approach we take in this paper.

Similar questions on the ergodic properties of the subadditive equilibrium states were also studied beyond the realm of irreducible locally constant cocycles. 
Within the class of fiber-bunched cocycles (see Subsection \ref{subsec: cocycles}), under suitable assumptions the second author \cite{park2019quasi} and Bochi and Garibaldi \cite{bochi2019extremal} established the uniqueness of the equilibrium states for the norm potentials. For such equilibrium states, the authors showed in \cite{call2020k} that total ergodicity implies the $K$-property (see Subsection \ref{subsec: ergodic properties} and Proposition \ref{prop: CP}). Our first result shows that these equilibrium states in fact have the stronger ergodic property of being Bernoulli. It extends our previous work and also generalizes existing results on locally constant cocycles to fiber-bunched cocycles.

\begin{thmx}\label{thm: A}
Let $\A \colon \Sig \to \glr$ be a fiber-bunched cocycle. 
\begin{enumerate}
\item
Suppose that $\A$ is strongly bunched and irreducible. If the unique equilibrium state $\mu_\A$ for the norm potential $\Phi_\A$ is totally ergodic, then it is Bernoulli.
\item
Supposing instead that $\A$ is typical, the unique equilibrium state $\mu_{\A,s} $ for the singular value potential $\Phi^s_\A$ is Bernoulli.
\end{enumerate}
\end{thmx}

In fact, we show that if the unique equilibrium state for a subadditive and quasi-multiplicative potential over a two-sided subshift is totally ergodic, then it is Bernoulli, and Theorem \ref{thm: A} describes specific examples that satisfy such assumptions. Indeed, these assumptions imply that the equilibrium state has the $K$-property and measurable local product structure (see Remark \ref{rem: 1} and Definition \ref{defn: lps}), and Bernoullicity is then established as a consequence of the following theorem; see Remark \ref{rem: thm A as cor}.

\begin{thmx}\label{thm: symbolic K to Bernoulli}
Let $(\Sigma_T,\sigma,\mu)$ be a mixing shift of finite type with the $K$-property. If $\mu$ is non-atomic and there exists $C > 0$ such that for all words $\I,\J$,
$$\mu([\I\J]) \leq C\mu([\I])\mu([\J]),$$
then $(\Sigma_T,\sigma,\mu)$ is Bernoulli.
\end{thmx}

This condition on $\mu$ is a natural condition to impose to obtain the Bernoulli property. When $\mu([\I\J]) \geq C^{-1}\mu([\I])\mu([\J])$ as well, this condition has been shown to imply the Bernoulli property without the assumption of the $K$-property \cite{Walters05}.

In this paper, it follows from a more general theorem that provides conditions under which the $K$-property implies Bernoullicity for Smale spaces, or metric Anosov systems. These systems are generalizations of Anosov systems beyond the smooth setting, and in particular, encompass shifts of finite type.

\begin{thmx}\label{thm: B}
Let $\mu$ be a fully supported, ergodic, measure on a metric Anosov system $(X,f)$. If $\mu$ satisfies the $K$-property and furthermore, is absolutely continuous with respect to a non-atomic product measure on the local stable and unstable sets on an open neighborhood, then $\mu$ is Bernoulli.
\end{thmx}

Note that it is possible to drop the fully supported condition, provided one can show local product structure of $\mu$ at almost every $x\in X$, not just a single open set. A close variant of this result is possibly expected by experts; in particular, when the measure in question is equivalent to a non-atomic product measure. However, we believe that the proof has not previously been written down. Indeed, this result is similar in flavor to many existing results establishing Bernoullicity from the $K$-property under various additional assumptions (see also \cite{ledrappier2016ergodic, ponce2018bernoulli} for related results on Bernoullicity in different settings, as well as the excellent book \cite{ponce2019introduction} on settings in which $K$ implies Bernoulli). It is always necessary to have these assumptions, as there are systems that are $K$ but not Bernoulli \cite{ornstein1973KnotBernoulli, katok1980smooth, kanigowski2018non}. In our work, we will refer to our main additional assumption as \textit{measurable local product structure}. We define this to mean that on open neighborhoods, the measure is absolutely continuous with respect to a non-atomic product measure on local stable and unstable sets; see Definition \ref{defn: lps} for the precise definition. After the initial version of this paper was posted, Alansari posted work \cite{alansari2022ergodic} addressing similar questions. There are two major differences. We consider non-smooth systems, but require a version of uniform hyperbolicity, while Alansari considers only smooth systems, but allows non-uniform hyperbolicity. It seems likely that these two approaches could be fruitfully combined.

\begin{rem}\label{rem: 1}
We stress that in this paper, we only need our measure to be absolutely continuous, not equivalent, to a product measure for us to say it has measurable local product structure. Furthermore, we include the word \textit{measurable} to indicate that the Radon-Nikodym derivative need not be continuous.
\end{rem}

Our proof follows the classical approach of Chernov and Haskell \cite{chernov1996nonuniformly} and Ornstein and Weiss \cite{ornstein1973geodesic, ornstein1998bernoulli} in that we show the existence of Very Weak Bernoulli partitions. The major difference is that we do not require our measure to be smooth. While this approach is not new, applying it to subadditive equilibrium states provides a different perspective in studying ergodic properties of such measures beyond locally constant cocycles. Moreover, we have simplified the relevant parts of the proof according to the fact that the base dynamical system is uniformly hyperbolic, as opposed to non-uniformly hyperbolic in \cite{chernov1996nonuniformly}.

The paper is organized as follows. In $\S$\ref{sec: prelim}, we introduce relevant notations and survey preliminary results. In $\S$\ref{sec: LPS}, we establish measurable local product structure on shift-invariant measures with a submultiplicative property, and in $\S$\ref{sec: Bernoulli} we prove an abstract theorem on Bernoullicity which implies Theorem \ref{thm: B}.
\\\\
\noindent\textbf{Acknowledgments:} We would like to thank Dan Thompson and Amie Wilkinson for helpful discussion and valuable comments on this paper. We would also like to thank the anonymous reviewers for extremely helpful comments. This work was partially supported by NSF grants DMS-1954463 and DMS-2303333, as well as a Presidential Fellowship from The Ohio State University.

\section{Preliminaries}\label{sec: prelim}

\subsection{Metric Anosov systems}\label{subsec:metric Anosov}
Smale spaces, or metric Anosov systems, are metric analogues of Anosov diffeomorphisms. In particular, they include shifts of finite type, to which we will apply our results later. We do not fully develop the theory in this section, instead referring to \cite{Putnam} for more information.

\begin{defn}
	Let $(X,d)$ be a compact metric space and $f : X\to X$ a homeomorphism. $(X,d,f)$ is a metric Anosov system if there exist constants $\eps_X > 0$, $0 < \chi < 1$, and a continuous bracket operation $[\cdot,\cdot] : \{(x,y) \mid d(x,y) \leq \eps_X\}\to X$ with the following properties:
	\begin{enumerate}
		\item $[x,x] = x$
		\item $[x,[y,z]] = [x,z]$ (when this is defined)
		\item $[[x,y],z] = [x,z]$ (when this is defined)
		\item $[f(x),f(y)] = f([x,y])$ (when this is defined)
		\item If $[y,x] = x = [z,x]$, then $d(f(y),f(z)) \leq \chi d(y,z)$
		\item If $[x,y] = x = [x,z]$, then $d(f^{-1}(y),f^{-1}(z)) \leq \chi d(y,z)$.
	\end{enumerate}
\end{defn}

Heuristically, one should think of the bracket operation as encapsulating the local product structure of Anosov systems, in which we can write $[x,y] = W_{\operatorname{loc}}^s(x)\cap W_{\operatorname{loc}}^u(y)$ as the intersection of local stable and unstable manifolds. With this in mind, we introduce the analogue of local stable and unstable manifolds.

\begin{defn}
For all $0 < \varepsilon \leq \eps_X$, define the local stable and unstable sets by
	$$X^s(x,\varepsilon) = \{y\mid d(y,x) \leq \varepsilon, [y,x] = x\}$$
	and
	$$X^u(x,\varepsilon) = \{y\mid d(y,x) \leq \varepsilon, [x,y] = x\}.$$
\end{defn}

With this definition, one can see that for all $y\in X^s(x,\eps_X)$, $d(f^ny,f^nx) \leq \chi^nd(x,y)$, and for $y\in X^u(x,\eps_X)$, $d(f^{-n}y,f^{-n}x) \leq \chi^nd(x,y)$. Furthermore, these behave nicely with the bracket operation, and there exists $\eps' > 0$ such that $[\cdot, \cdot]$ is a homeomorphism onto its image from $X^s(x,\eps')\x X^u(x,\eps')$. Moreover, the bracket operation behaves much as we would hope with respect to $f$.

\begin{prop}\label{prop: images of stables}
	For all $\varepsilon \leq \frac{\eps_X}{2}$,
	$$(f\times f)(X^s(x,\varepsilon)\x X^u(y,\chi\varepsilon)) \subset X^s(fx,\varepsilon)\x X^u(fy,\varepsilon)$$
	and
	$$(f\times f)^{-1}(X^s(x,\chi\varepsilon)\x X^u(y,\varepsilon))\subset X^s(f^{-1}x,\varepsilon)\x X^u(f^{-1}y,\varepsilon).$$
\end{prop}

\begin{proof}
Given $w\in X^s(x,\varepsilon)$, we observe first that $d(f(w),f(x)) \leq d(w,x)$, and so we can compute the following bracket operation:
$$[f(w),f(x)] = f([w,x]) = fx.$$
Therefore, $f(w)\in X^s(fx,\varepsilon)$. Similar arguments give the rest of the necessary inclusions.
\end{proof}

The final fact that we will need is that the bracket operation is still defined on the image of a set under $f^n$ for $\abs{n}$ large, provided that the diameter of the original set is appropriately small.

\begin{prop}\label{prop: rectangles under pushforwards}
For all $n\in \mathbb{Z}$, there exists $\eps_n > 0$ such that the bracket operation is well-defined on $f^nX^s(x,\varepsilon)\times f^nX^u(x,\varepsilon)$ for all $\varepsilon \leq \eps_n$.
\end{prop}

\subsection{Symbolic dynamical systems}

Let $T$ be a $q \times q$ square-matrix with entries in $\{0,1\}$. We define $\Sig \subset \{1,\cdots, q\}^\Z$ to be the set of all bi-infinite sequences $(x_i)_{i\in \Z}$ satisfying $T_{x_i,x_{i+1}}=1$ for all $i\in \Z$. The left shift map $\s\colon \Sig \to \Sig$ is defined as $(\s x)_i = x_{i+1}$ for $x = (x_i)_{i\in \Z}$ and $i\in \Z$. 
We denote the corresponding one-sided subshift by $\Sig^+ \subset  \{1,\cdots, q\}^{\N_0}$. The left shift map, again denoted by $\s$, is defined analogously with $i \in \N_0$. We denote the one-sided subshift on negative symbols as $\Sig^- \subset \{1,\cdots, q\}^{-\N_0} $ with the shift map given by $\s^{-1}$.

An \textit{admissible word of length $n$} is a word $i_0\ldots i_{n-1}$ with $i_j \in \{1,\ldots,q\}$ such that $T_{i_j,i_{j+1}} = 1$ for all $0 \leq j \leq n-2$.
Let $\L$ be the collection of all admissible words. For $\I\in \L$, we denote its length by $|\I|$.
For each $n \in \N$, let $\L(n) \subset \L$ be the set of all admissible words of length $n$. For any $\I =  i_0\ldots i_{n-1} \in \L(n)$, we define the associated \textit{cylinder} by
$$[\I]=[ i_0\ldots i_{n-1}]:=\{y \in \Sig \colon y_j = i_j \text{ for all } 0 \leq j \leq n-1\}.$$
For any $x=(x_i)_{i \in \Z}\in \Sig$ and $n\in \N$ we denote by $[x]_n:=[x_0\ldots x_{n-1}]$ the $n$-cylinder containing $x$.

We endow $\Sig$ with the metric $d$ defined as follows: for $x = (x_i)_{i \in \Z},y = (y_i)_{i \in \Z} \in \Sig$, we have
$$d(x,y)  = 2^{-k},$$ 
where $k$ is the largest natural number such that $x_i = y_i$ for all $|i| < k$. This makes $\Sig$ into a compact metric space, and $(\Sig, \s)$ into a Smale space with $\eps_{\Sigma_T} = 1$.

Given $x\in\Sig$, the \textit{local stable set} of $x$ at scale $2^{-k}$ for $k \in\mathbb{N}$ is
$$X^s(x,2^{-k}) = \{y\in\Sig\mid x_i = y_i \text{ for } i\geq -k \}$$
and analogously, the \textit{local unstable set} of $x$ at scale $2^{-k}$ is
$$X^u(x,2^{-k}) = \{y\in\Sig\mid x_i = y_i \text{ for } i\leq k\}.$$
When $k = 0$, we will write $\Wloc^s(x) := X^s(x,1)$ and $\Wloc^u(x) := X^u(x,1)$.

Finally, we will always assume that the adjacency matrix $T$ is \textit{primitive}, meaning that there exists $N>0$ such that all entries of $T^N$ are positive. The primitivity of $T$ is equivalent to $(\Sig,\s)$ being topologically mixing.

\subsection{Subadditive thermodynamic formalism}\label{subsec: symbolic}
Let $\Phi = \{\log \vp_n\}_{n\in \N}$ be a sequence of continuous real-valued functions on $\Sig$. We say $\Phi$ is \textit{subadditive} if 
$$\log \vp_{m+n} \leq \log \vp_n + \log \vp_m \circ \s^n$$
for every $m,n\in \N$. Often we will refer to $\Phi$ as a \textit{subadditive potential} on $\Sig$.

For every subadditive potential $\Phi$, we can define its subadditive pressure $P(\Phi)$; we refer the readers to \cite{brin2002introduction} for its definition and further discussion. While this can be defined without reference to any measures, it has been shown to have an equivalent formulation. This is known as the \textit{subadditive variational principle} established as \cite[Theorem 1.1]{cao2008thermodynamic}:
$$P(\Phi)=\sup\limits_{\mu \in \M(\s)} \Big(h_\mu(\s)+ \lim\limits_{n\to \infty} \frac{1}{n}\int \log \vp_n \,d\mu \Big)$$
In this formulation, $\M(\sigma)$ denotes the set of $\sigma$-invariant probability measures on $X$, and any $\mu\in\M(\sigma)$ attaining the supremum is called an \textit{equilibrium state}.

We say that a probability measure $\mu$ on $\Sig$ satisfies the \textit{Gibbs property} with respect to $\Phi$ if there exists a constant $C\geq 1$ such that for any $x\in \Sig$ and $n\in \N$, we have
$$C^{-1} \leq \frac{\mu([x]_n)}{e^{-nP(\Phi)} \vp_n(x)} \leq C.$$

We now introduce properties that will be satisfied by all subadditive potentials considered in this paper. First, we say $\Phi$ has \textit{bounded distortion} if there exists $C \geq 1$ such that for any $x\in \Sig$, $n\in\N$, and $y \in [x]_n$, we have
$$C^{-1} \leq \frac{\vp_n(x)}{\vp_n(y)} \leq C,$$
which going forward we will write as $\frac{\vp_n(x)}{\vp_n(y)}\asymp C$. 

\begin{rem}\label{rem: identification}
For each $\I\in \L(n)$, we define 
$$\vp(\I):=\sup\limits_{x \in [\I]} \vp_n(x).$$
Using this, $\log \vp$ defines a subadditive potential on $\L$ (i.e., $\log\vp(\I\J) \leq \log\vp(\I) + \log\vp(\J))$ for $\I,\J\in\L$. By an abuse of notation, we shall denote this again by $\Phi$. 
Conversely, given a subadditive potential $\Phi:=\log \vp$ on $\L$, we can induce on $\Sig$ a subadditive potential, denoted again by $\Phi$, by setting $\vp_n(x) := \vp([x]_n)$. 

We will identify such potentials (on $\Sig$ and on $\L$) with one another; this identification will simplify the presentation of the preliminary results described below.
Moreover, relevant objects such as the subadditive pressure and the set of equilibrium states remain unchanged under this identification if the original subadditive potential $\Phi = \{\log \vp_n\}$ on $\Sig$ has bounded distortion; see \cite{park2019quasi}.
\end{rem}

The last property we will make use of is called quasi-multiplicativity. We say a subadditive potential $\Phi=\log \vp$ on $\L$ is \textit{quasi-multiplicative} if there exist $c>0$ and $k\in \N$ such that for any $\I,\J \in \L$ there exists $\K \in \L$ with $|\K| \leq k$ so that
$$ \vp(\I\K\J) \geq c \vp(\I)\vp(\J).$$
 The following proposition whose proof can be found in \cite[Theorem 5.5]{feng2011equilibrium} guarantees that quasi-multiplicativity is a sufficient condition for the uniqueness of the equilibrium state.
  
\begin{prop}\label{prop: qm}
Let $\Phi$ be a subadditive potential on $\L$. If $\Phi$ is quasi-multiplicative, then it has a unique equilibrium state $\mu \in \M(\s)$. Moreover, $\mu$ has the Gibbs property with respect to $\Phi$.
\end{prop}

\subsection{Ergodic properties of an invariant measure}\label{subsec: ergodic properties}
In this subsection, let $(X,f,\mathcal{B})$ be a homeomorphism on a compact metric space. We provide a brief introduction to the various mixing properties we use throughout this paper.

Given a Borel invariant measure $\mu \in \M(f)$, we say $\mu$ is \textit{totally ergodic} if it is ergodic with respect to $f^n$ for all $n\in \N$. Total ergodicity is stronger than ergodicity, but weaker than mixing.

A stronger mixing property is the Kolmogorov property (the \textit{$K$-property}, for short), which is stronger than mixing of all orders but weaker than Bernoulli. It has many equivalent formulations, and the following definition (known as \emph{completely positive entropy}) is an equivalence established by Rokhlin and Sinai.
 
\begin{defn}\label{defn: K}
An invariant measure $\mu \in \M(f)$ has the \textit{$K$-property} if its Pinsker factor (i.e., maximal zero entropy factor) is trivial.
\end{defn}
The authors have shown in a previous paper that under the conditions below, total ergodicity implies the $K$-property:

\begin{prop}\cite{call2020k}\label{prop: CP}
	Let $\Phi$ be a subadditive potential on $\Sig$, and suppose it is quasi-multiplicative and has bounded distortion. If its unique equilibrium state is totally ergodic, then it has the $K$-property.
\end{prop}

In this paper, the goal is to show that under these hypotheses, the underlying measure is Bernoulli.
In order to show this implication, we will make use of another one of the equivalent formulations for the $K$-property as described in the following proposition; see \cite[Section 2.8]{ponce2019introduction}. In order to state it, we first need an additional definition.

\begin{defn}\label{defn: eps almost everywhere}
Given a measurable partition $\xi$ of $X$ and a probability measure $\mu$ on $X$, a property holds \textit{$\ep$-almost everywhere} on $\xi$ if the union of atoms on which the property does not hold has $\mu$-measure at most $\ep$.
\end{defn}

\begin{prop}\label{prop: K-partitions}
For every finite measurable partition $\xi$ of a $K$-system $(X,\mathcal{B},\mu,f)$, for all $\ep > 0$ and finite collection of sets $A_i\in \mathcal{B}$, there exists $N\in \N$ such that for all $m_2 \geq m_1 \geq N$, for $\ep$-almost every $E\in \bigvee_{m_1}^{m_2}f^i\xi$ and each $A_i$, 
$$\left|\frac{\mu(E\cap A_i)}{\mu(E)} - \mu(A_i)\right| \leq \ep.$$
\end{prop}

Finally, we give a formal definition of the Bernoulli property.
\begin{defn}\label{defn: B}
 Given a finite set $S$ equipped with a probability measure $\nu$, we say the dynamical system $(S^{\mathbb{Z}},\s,\nu^{\mathbb{Z}})$ is a \textit{Bernoulli shift}. We say that $(X,f,\mathcal{B},\mu)$, where $\mu\in\M(f)$ is an invariant measure, is \textit{Bernoulli} if $(X,f,\mathcal{B},\mu)$ is measurably isomorphic to a Bernoulli shift.
\end{defn}

\subsection{Matrix cocycles}\label{subsec: cocycles}

Throughout the paper, we will assume that the norm on $\glr$ is the operator norm $\|\cdot \|$.
Let $\A \colon \Sig \to \glr$ be an $\alpha$-\hol continuous $\glr$-valued function. Such a function induces a matrix cocycle (which we also denote by) $\A \colon \Sig \times \N \to \glr$ where
$$\A(x,n):=\A(\s^{n-1}x) \ldots \A(x).$$
We will often denote $\A(x,n)$ by $\A^n(x)$, and such a cocycle $\A$ satisfies a \textit{cocycle identity} $\A(x,n+m) = \A(\s^nx,m)\A(x,n)$ for all $m,n \in \N$. Using another function $\A^{-1}$ over $(\Sig,\s^{-1})$ defined by $$\A^{-1}(x):=\A(\s^{-1}x)^{-1},$$
we can likewise define 
$\A^{-n}(x)$, which coincides with $\A^n(\s^{-n}x)^{-1}$, and extend our cocycle $\A$ over $\Sig  \times \Z$, satisfying the cocycle identity for all $m,n \in \Z$.

From the submultiplicativity of the norm $\|\cdot \|$, the sequence of functions $\Phi_\A:=\{\log \vp_{\A,n}\}$ given by $$\vp_{\A,n}(x):=\|\A^n(x)\|$$ defines a subadditive potential. We call this the \textit{norm potential} of $\A$.
\begin{rem}
Often in the literature, the norm potential refers to a family of subadditive potentials defined as $\{\log \|\A^n(\cdot)\|^t\}_{n\in \N}$ for $t>0$. Under the same assumption on $\A$ from Theorem \ref{thm: A} (1), the same proof can be applied to show that the unique equilibrium state for the potential $\{\log \|\A^n(\cdot)\|^t\}_{n\in \N}$ is Bernoulli.
\end{rem}

We denote the singular values of $A \in \glr$ by $\s_1(A) \geq \ldots \geq \s_d(A)>0$. For $s \in [0,d]$, we define
$$\vp^s(A):= \s_1(A)\ldots \s_{\lfloor s \rfloor}(A)\s_{\lceil s \rceil}(A)^{\{s \}},$$
whereas for $s > d$ we set $\vp^s(A):=|\det(A)|^{s/d}$. 
For any $\A \colon \Sig \to \glr$ and $s \geq 0$, we define the $s$-\textit{singular value potential} as $$\Phi_\A^s:=\{\log \vp_{\A,n}^s\}_{n\in \N} \text{ where } \vp_{\A,n}^s :=\vp^s(\A^n(x)).$$
Since $\|A\| = \s_1(A)$, the $1$-singular value potential $\Phi_\A^1$ coincides with the norm potential $\Phi_\A$. Therefore, singular value potentials are generalizations of norm potentials.

When $\A\ \colon \Sig \to \glr$ depends only on the zero-th symbol $x_0$ of $x=(x_i)_{i\in \N}$, we say that it is \textit{locally constant}. 
When $\A$ is \textit{irreducible}, meaning that there does not exist a proper subspace $V \subset \R^d$ preserved under the action of $\A$, its norm potential $\Phi_\A$ has a unique equilibrium state. In this case, Morris has shown in \cite{morris2018ergodic, morris2019necessary, morris2020totally} that for such equilibrium states total ergodicity implies Bernoullicity.

In this paper we expand Morris' result on Bernoullicity to a class of cocycles that are not necessarily locally constant. Moreover, we apply a fundamentally different technique than Morris, relying on an approach pioneered by Ornstein and Weiss \cite{ornstein1973geodesic, ornstein1998bernoulli}; see Section \ref{sec: Bernoulli} for more details. We will focus on the class of fiber-bunched cocycles, which we define below. 

We say an $\a$-\hol cocycle $\A \colon \Sig \to \glr$ is \textit{fiber-bunched} if
\begin{equation}\label{eq: fb}
\|\A(x)\|\cdot \|\A(x)^{-1}\| <2^\a
\end{equation}
for every $x\in \Sig$. It is clear from the definition that the set of fiber-bunched cocycles forms an open subset of the set of $\a$-\hol cocycles $C^\a(\Sig,\glr)$ and that perturbations of conformal cocycles are fiber-bunched cocycles. In fact, fiber-bunched cocycles may as well be thought of as nearly conformal cocycles.  

The fiber-bunching assumption guarantees the convergence of the canonical holonomies: for $y \in \Wloc^s(x)$, we define
$$H^s_{x,y}:=\lim\limits_{n\to \infty} \A^n(y)^{-1}\A^n(x).$$
The collection $\{H^s_{x,y}\}\subseteq \glr$ is called the \textit{canonical stable holonomy}. The holonomy satisfies the following properties:  for any $y,z \in \Wloc^s(x)$,
\begin{enumerate}
\item $H^s_{x,x} = \id$ and $H^s_{y,z}H^s_{x,y} = H^s_{x,z}$,
\item $\A(y)H^s_{x,y} = H^s_{\s x,\s y }\A(x)$,
\item $(x,y) \mapsto H^s_{x,y}$ is continuous,
\item $\|H^s_{x,y} -\id \| \leq C \cdot d(x,y)^\a$.
\end{enumerate}
The second property says that the holonomy is $\A$-equivariant, and using this property we can define \textit{global stable holonomy} $H^{s}_{x,y}$ for $y \in \W^s(x)$ that does not necessarily lie on $\Wloc^s(x)$. The fourth property says that the stable holonomies are \hol continuous with the same exponent as the cocycle.

The \textit{canonical unstable holonomy} $\{H^u_{x,y}\}$ can likewise be defined using $\A^{-1}$:
$$H^u_{x,y} := \lim\limits_{n\to\infty} \A^{-n}(y)^{-1}\A^{-n}(x).$$
The unstable holonomy satisfies the analogous properties as above with $s$ and $\s$ replaced by $u$ and $\s^{-1}$, respectively. See \cite{kalinin2013cocycles} for a more detailed discussion on the canonical holonomies of fiber-bunched cocycles.

\begin{rem}\label{rem: bdd distortion}
For any fiber bunched cocycle $\A$, the \hol continuity of $H^{s/u}$ implies that the norm potential $\Phi_\A$ and the singular value potential $\Phi_\A^s$ for any $s \geq 0$ have bounded distortion; see \cite{park2019quasi}. This justifies the identification assumed in Remark \ref{rem: identification}.
\end{rem}

We now introduce the irreducibility and the strong bunching assumptions for fiber-bunched cocycles appearing in the first case of Theorem \ref{thm: A}. In the following definition, a bundle $\V$ over $\Sig$ is a continuous map from $\Sig$ into the Grassmannian $Gr(d,r)$ for some $0 < r \leq d$, and a subbundle $\mathbb{W}$ of $\V$ is a continuous map from $\Sig$ to $Gr(d,s)$ for some $0 < s \leq r$ such that $\mathbb{W}(x) \subseteq \V(x)$ for every $x\in \Sig$.
The \textit{bi-holonomy invariance} of a bundle means that $\V_y = H^s_{x,y}\V_x$ for any $y \in \Wloc^s(x)$, and likewise along the unstable holonomy. As the holonomies are $\a$-\hol continuous, any bi-holonomy invariant bundle is also necessarily $\a$-\hol continuous.

\begin{defn}
We say a fiber-bunched cocycle $\A \colon \Sig \to \glr$ is \textit{irreducible} if there does not exist an $\A$-invariant and bi-holonomy invariant subbundle. 
\end{defn}

This definition of irreducibility for fiber-bunched cocycles is compatible with irreducibility for locally constant cocycles: when a fiber-bunched cocycle $\A$ is locally constant, then the canonical holonomies are trivial (i.e., $H^{s/u} \equiv \id$), and hence, the two notions of irreducibility coincide.

\begin{defn}
We say $\A \colon \Sig \to \glr$ is \textit{strongly bunched} if
\begin{enumerate}
\item $d=2$ and $\A$ is fiber-bunched as in \eqref{eq: fb}, or
\item $d\geq 3$ and for every $x\in \Sig$,
$$\|\A(x)\|\cdot \|\A(x)^{-1}\| <2^{\a/3}.$$
\end{enumerate} 
\end{defn}

Notice that when $d=2$, strong bunching coincides with fiber-bunching. It is only when $d\geq 3$ the strong bunching requires a more stringent bound compared to \eqref{eq: fb}.

\begin{rem}The strong bunching assumption was first introduced in Bochi and Garibaldi \cite{bochi2019extremal} in the more general setting of hyperbolic homeomorphisms. The general definition there requires that the product $\|\A(x)\|\|\A(x)^{-1}\|$ is bounded above by a constant depending only on the base dynamical system and the exponent of the cocycle. In the case of a subshift of finite type $(\Sig,\s)$ equipped with the standard metric defined in Subsection \ref{subsec: symbolic}, they showed that such a bound can be taken to be $2^{\a/3}$. 

\end{rem}

We now introduce the typicality assumption appearing in the second case of Theorem \ref{thm: A}.

\begin{defn}Let $\A \colon \Sig \to \glr$ be a fiber-bunched cocycle. \\We say $\A$ is \textit{1-typical} if
\begin{enumerate}
\item there exists a periodic point $p\in \Sig$ such that $P:=\A^{\text{per}(p)}(p)$ has simple eigenvalues of distinct moduli, and
\item there exists a homoclinic point $z \in \W^s(p) \cap \W^u(p) \setminus \{p\}$ such that the holonomy loop $\psi:=H^s_{z,p}H^u_{p,z}$ maps the eigendirections $\{v_1,\ldots,v_d\}$ of $P$ into a general position. 
\end{enumerate}
Moreover, we say $\A$ is \textit{typical} if the exterior product cocycle $\A^{\wedge t}$ is 1-typical for every $1 \leq t \leq d-1$.
\end{defn}

Typicality was first introduced by Bonatti and Viana \cite{bonatti2004lyapunov}. They showed that the Lyapunov exponents of typical cocycles with respect to the unique equilibrium states for \hol potentials are simple. Two assumptions appearing in the definition of typicality are called \textit{pinching} and \textit{twisting}, and their roles parallel those of proximality and strong irreducibility in the classical setting of i.i.d. product of random matrices. 
We also note that typicality is a stronger assumption which implies irreducibility.

The following proposition summarizes the known results on the singular value potentials of the cocycles considered in Theorem \ref{thm: A}.

\begin{prop}\label{prop: cocycles}
 Let $\A \colon \Sig \to \glr$ be a fiber-bunched cocycle. 
\begin{enumerate}
\item If $\A$ is strongly bunched and irreducible, then its norm potential $\Phi_\A$ is quasi-multiplicative.
\item If $\A$ is typical, then its singular value potential $\Phi_\A^s$ is quasi-multiplicative for all $s \geq 0$.
\end{enumerate}
In both cases, Proposition \ref{prop: qm} gives a unique equilibrium state satisfying the Gibbs property.
\end{prop}
For the first statement, it was shown by Bochi and Garibaldi \cite{bochi2019extremal} that $\A$ is spannable. The second author and Butler \cite{butler2019thermodynamic} then showed that the norm potentials of spannable cocycles are quasi-multiplicative. The second statement is established by the second author in \cite{park2019quasi}. We also note that 1-typicality implies that the norm potential $\Phi_\A$ is quasi-multiplicative, and typicality ensures that the same is true for the $s$-singular value potentials $\Phi_\A^s$ for all $s \geq 0$.

Notice from its definition that if $\A \colon \Sig \to \glr$ is typical, then so are its powers $\A^n$ for all $n\in\N$. Hence, Proposition \ref{prop: cocycles} has an immediate corollary:
\begin{cor}\label{cor: typical then totally erg}
If $\A$ is typical, then for any $s\geq 0$, the unique equilibrium state $\mu_{\A,s}$ for $\Phi_\A^s$ is totally ergodic. 
\end{cor}
This corollary was already observed in the previous paper of the authors \cite{call2020k}. This is the reason why we do not have to separately assume the total ergodicity on $\mu_{\A,s}$ from the second case of Theorem \ref{thm: A}, because it is already implied by typicality.

\begin{rem}\label{rem: thm A as cor}
In view of Propositions \ref{prop: CP}, and \ref{prop: cocycles}, and Corollary \ref{cor: typical then totally erg}, the unique equilibrium states under consideration in Theorem \ref{thm: A} satisfy the $K$-property. Once we establish that such measures are also absolutely continuous with respect to a non-atomic product measure on the local stable and unstable sets (which we call \textit{measurable local product structure} in the next section), Theorem \ref{thm: A} follows from Theorem \ref{thm: symbolic K to Bernoulli}. The measurable local product structure of such measures is established in Proposition \ref{prop: LPS}. The proof of Theorem \ref{thm: symbolic K to Bernoulli} follows from a more general form of Theorem \ref{thm: B}, which appears in Section \ref{sec: Bernoulli}.
\end{rem}

\section{Measurable Local Product Structure}\label{sec: LPS}

\begin{defn}\label{defn: lps}
Let $(X,f)$ be a metric Anosov system, and let $\mathcal{B}$ be the Borel $\sigma$-algebra on $X$. A Borel measure $\mu$ on $X$ has \textit{measurable local product structure} if $\mu$-almost every $x\in X$ belongs to a positive measure open set $U \subset X^s\left(x,\frac{\eps_X}{2}\right) \x X^u\left(x,\frac{\eps_X}{2}\right)$ for which:
\begin{itemize}
    \item There exist non-atomic, complete, Borel probability measures $\nu,\rho$ on the spaces $\left(X^u\left(x,\frac{\eps_X}{2}\right),\mathcal{B}\cap X^u\left(x,\frac{\eps_X}{2}\right) \right)$ and $\left(X^s\left(x,\frac{\eps_X}{2}\right),\mathcal{B}\cap X^s\left(x,\frac{\eps_X}{2}\right)\right)$, respectively
    \item $\mu|_U\ll \nu\x\rho$.
\end{itemize}
\end{defn}
Note that these assumptions on $\nu$ make $(X^u(x,\varepsilon), \nu|_{X^u(x,\varepsilon)})$ equipped with the completed Borel $\sigma$-algebra a Lebesgue space for all $\varepsilon \leq \frac{\eps_X}{2}$, and similarly for $X^s(x,\eps)$ and $\rho$. Furthermore, $\nu\x\rho$ is a Borel probability measure on $X^u(x,\frac{\eps_X}{2})\x X^s(x,\frac{\eps_X}{2})$.

Of interest, though not necessary for the applications present in this paper, when working with a fully supported ergodic measure, this definition can be checked on a single open set.
\begin{prop}
If a Borel measure $\mu$ is ergodic and fully supported, then $\mu$ has measurable local product structure if there exists $x\in U\subset X^u\left(x,\frac{\eps_X}{2}\right)\x X^s\left(x,\frac{\eps_X}{2}\right)$ such that:
\begin{itemize}
    \item $\mu(U) > 0$ and $U$ is open,
    \item There exist non-atomic, complete, Borel probability measures $\nu,\rho$ on $X^u\left(x,\frac{\eps_X}{2}\right)$ and $X^s\left(x,\frac{\eps_X}{2}\right)$, respectively,
    \item $\mu|_U\ll \nu\x\rho$.
\end{itemize}
\end{prop}

\begin{proof}
Let $x,U, \rho, \nu$ be as in the proposition statement. Now, let $y\in U$ be arbitrary. For all sufficiently small $\delta > 0$, we can define a probability measure $\rho_{y,\delta}$ on $X^s(y,\delta)$ by taking $\rho_{y,\delta}(A) = \frac{\rho([x,A])}{\rho([x,X^s(y,\delta)])}$, and we define $\nu_{y,\delta}$ on $X^u(y,\delta)$ similarly. Note that in these definitions, we are implicitly using the fact that $\mu$ is fully supported. This allows us to assume that $\rho([x,X^s(y,\delta)]) > 0$, as if not, we would have a small ball containing $y$ of zero $\mu$-measure. With this established, let $n\in\N$ be arbitrary. Taking $\eps_n$ as in Proposition \ref{prop: rectangles under pushforwards}, for all $y\in U$, we have established that for some $\delta < \epsilon_n$, $\mu\ll \rho_{y,\delta}\x \nu_{y,\delta}$. Now, observe that $f^n X^s(y,\delta) \subset X^s(f^ny,\eps_X/2)$ and $f^nX^u(y,\delta) \subset X^u(f^ny,\eps_X/2)$ (see Proposition \ref{prop: images of stables}). Furthermore, we have
$$\mu = (f^n)^*\mu \ll (f^n)^*\rho_{y,\delta} \x (f^n)^*\nu_{y,\delta}.$$
Therefore, $f^n y$ is contained in an open set $U'\subset X^s(f^ny,\epsilon)\x X^u(f^ny,\epsilon)$ on which $\mu|_{U'}$ is absolutely continuous with respect to a product measure. Now $\mu$-almost every $z\in X$ is in $f^n U$ for some $n$ by ergodicity of $\mu$, and so we can apply the previous argument taking $y = f^{-n}z$, completing our proof.
\end{proof}

The following proposition shows that subadditive equilibrium states under consideration in Theorem \ref{thm: A} have measurable local product structure.

\begin{prop}\label{prop: LPS}
Let $(\Sigma_T,\sigma)$ be a mixing shift of finite type. If $\mu$ is a non-atomic, $\sigma$-invariant, complete, Borel measure such that $\mu(\I\J) \leq C\mu(\I)\mu(\J)$ for some $C > 0$ independent of $\I,\J\in\L$, then $\mu$ has measurable local product structure.
\end{prop}

\begin{proof}
Recall that $(\Sigma_T,\sigma)$ is a shift of finite type which we have chosen to have forbidden words of length at most $2$ (i.e., to check the admissibility of a word, we need only confirm that there are no inadmissible words of length $2$). Thus, given any $x,y\in\Sigma_T$ with $x_0 = y_0$, we can define $z\in\Sigma_T$ by $z_n = x_n$ for $n\leq 0$ and $z_n = y_n$ for $n\geq 0$.

Let $x\in\Sigma_T$ be such that $\mu([x_{-1}x_0x_1]) > 0$. This encompasses $\mu$-almost every $x\in\Sigma_T$. Now, $\sigma^{-1}[x_{-1}x_0x_1]$ can be identified with the product $X^s(x,\frac{1}{2})\x X^u(x,\frac{1}{2})$ via the bracket operation discussed in \ref{subsec:metric Anosov}. We will define measures $\mu^+$ and $\mu^-$ on $X^u(x,\frac{1}{2})$ and $X^s(x,\frac{1}{2})$, respectively. Letting $\pi_u : \Sigma_T\to \Sigma_T^+$ be projection, define $\mu^+$ on $X^u(x,2^{-1})$ by
$$\mu^+(A) := \frac{\mu(\pi_u^{-1}(\pi_u(A\cap X^u(x,2^{-1}))))}{\mu([x_0x_1])}$$
for any compact $A\subset \Sigma_T$. Now, define $\mu^-$ on $X^s(x,2^{-1})$ by
$$\mu^-(A) := \frac{\mu(\sigma(\pi_s^{-1}(\pi_s(\sigma^{-1}(A\cap X^s(x,2^{-1}))))))}{\mu([x_{-1}])}$$
This suffices to show that any Borel set is measurable, and so these measures are Borel.
The following straightforward argument will establish that they are also probability measures, as well as demonstrating necessary facts relating $\mu$ to $\mu^+\x\mu^-$.

Now, given $\I\in\L$, we will show that
$$\pi_u^{-1}(\pi_u([x_0x_1\I]\cap X^u(x,2^{-1}))) = [x_0x_1\I].$$
First, suppose $y\in [x_0x_1\I]$. Then define $z\in\Sigma_T$ by $z_n = x_n$ for $n\leq 0$ and $z_n = y_n$ for $n > 0$. From this definition, we see that $z\in [x_0x_1\I]\cap X^u(x,2^{-1})$ and $\pi_u(z) = \pi_u(y)$. The reverse inclusion is immediate. Therefore, we see that
$$\mu^+([x_0x_1\I]\cap X^u(x,2^{-1})) \geq \mu([x_0x_1\I]).$$
We show that a similar result holds for $\mu^-$ by analogous arguments. Namely, let $\J\in\L$, and observe that
$$\sigma^{1+|\J|}[\J x_{-1}] = \sigma(\pi_s^{-1}(\pi_s(\sigma^{|\J|}[\J x_{-1}]\cap \sigma^{-1}X^s(x,2^{-1})))).$$
This is because given $y\in \sigma^{1+|\J|}[\J x_{-1}]$, we can define $z\in\Sigma_T$ by taking $z_n = y_{n-1}$ for $n\leq 0$ and $z_n = x_{n-1}$ for $n\geq 0$. Then it is straightforward to check that $\pi_s(\sigma^{-1}y) = \pi_s(z)$ and that $z\in \sigma^{|\J|}[\J x_{-1}]\cap \sigma^{-1}X^s(x,2^{-1})$ as desired. Taking $\J$ to be the empty word shows that $\mu^-$ is a probability measure.

Together, these results show that for any cylinder set
$$\sigma^{1+|\I|}[\I x_{-1}x_0x_1\J]\subset [x_{-1}x_0x_1],$$
we have 
$$\mu(\sigma^{1+|\I|}[\I x_{-1}x_0x_1\J]) \leq C\mu([\I x_{-1}])\mu([x_0x_1\J]) \leq C\mu^-(\sigma^{1+|\I|}[\I x_{-1}]))\mu^+([x_0x_1\J]).$$
Consequently, $\mu\ll\mu^+\x\mu^-$, as desired.

The only remaining property needed is that $\mu^+$ and $\mu^-$ are non-atomic. This is immediate from the lack of atoms in $\mu$. If $\mu^+(\{y\}) > 0$ for some $y\in \Sigma_T$, this would imply that the set $A = \{z\in\Sigma_T \mid z_n = y_n \forall n\geq 0\}$ has positive $\mu$-measure. However, $\operatorname{diam}\sigma^{m}A \to 0$ as $m\to\infty$, and so $\mu(\sigma^mA)$ must go to $0$ because $\mu$ is non-atomic. As $\mu$ is measure-preserving, it follows immediately that $\mu(A) = 0$, as desired. Non-atomicity of $\mu^-$ is proved similarly.\end{proof}

Observe that shift-invariant measures with the Gibbs (or subadditive) Gibbs property satisfy this condition. Therefore, we can apply this result to the equilibrium states considered in Theorem \ref{thm: A}.

\section{Bernoulli property}\label{sec: Bernoulli}
The goal of this section is to state and get started on the proof of the following abstract theorem on Bernoulli property. Notice that it is a reformulation of Theorem \ref{thm: B} using measurable local product structure. As pointed out in Remark \ref{rem: thm A as cor}, Theorem \ref{thm: A} then also follows from it. 

\begin{thm}\label{thm: abstract}
Suppose $(X,f)$ is a Smale space. If $\mu$ is an invariant, complete, Borel measure which has measurable local product structure, then if $\mu$ has the $K$-property, it is Bernoulli.
\end{thm}

In this section, we will collect relevant preliminary results to prove this theorem, and the and the actual proof will appear over the next two sections.

\begin{note}
In this and the following sections, we will often have to compare two measures $\mu,\nu$ on the same space. When we do so, we will choose one representative of the Radon-Nikodym derivative and treat it as a well-defined function. Furthermore, we will often refer to bounds on the derivative. When we do, we will abuse notation and speak of bounds apart from sets of measure zero. Thus, for instance, if we say $\abs{\frac{d\mu}{d\nu}|_A - 1} < \ep$, this serves only as shorthand for the fact that for any measurable $B\subset A$ with $\nu(B) > 0$, $\abs{\frac{\mu(B)}{\nu(B)} - 1} < \ep$.
\end{note}

\subsection{Definitions}

For further details about the following definitions, we refer the reader to the book by Ponce and Varao \cite{ponce2019introduction}, as well as the paper by Chernov and Haskell \cite{chernov1996nonuniformly}.

Consider two measure spaces $(X,\mu)$ and $(Y,\nu)$. A probability measure $\lambda$ on $X \times Y$ is a \textit{joining} if $(\pi_1)_*\lambda = \mu$ and $(\pi_2)_*\lambda = \nu$ where $\pi_1,\pi_2$ are the canonical projections onto $X$ and $Y$, respectively. We denote by $J(\mu,\nu)$ the set of all joinings of $\mu$ and $\nu$. 

We now introduce the notion of $\ol{d}$-distance between partitions and sequences of partitions. Consider finite partitions $\xi=\{A_1,\ldots,A_k\}$ of $X$ and $\eta= \{B_1,\ldots,B_k\}$ of $Y$. For any $x\in X$, we define $\xi(x)$ as the unique index $j \in \{1,\ldots,k\}$ such that $x \in A_j$. Similarly, we define $\eta(y)$ for $y\in Y$. Then their $\ol{d}$-distance can be defined as follows:
$$\ol{d}(\xi,\eta):=\inf\limits_{\lambda \in J(\mu,\nu)} \lambda\{(x,y) \colon \xi(x) \neq \eta(y)\}.$$ 

This definition of the $\ol{d}$-distance can be extended to measure distance between two sequences of partitions $\{\xi_i\}_1^n$ on $X$ and $\{\eta_i\}_1^n$ on $Y$. First, for any $x\in X$ and $y\in Y$ we set
$$h(x,y): = \frac{1}{n}\sum_{i=1}^{n}\delta\{\xi_i(x) \neq \eta_i(y)\}$$
where $\delta$ is the indicator function. Using this function $h$, we define 
$$
\ol{d}(\{\xi_i\}_1^n,\{\eta_i\}_1^n) := \inf\limits_{\lambda \in J(\mu,\nu)} \int h(x,y) \, d\lambda(x,y).
$$

In many cases, the space $Y$ will be a $\mu$-measurable subset $E\subset X$ equipped with the conditional measure $\mu|_E$ defined by
$$\mu|_E(F) := \mu(F)/\mu(E)$$
for any measurable $F \subseteq E$. Moreover, the partition on $E \subset X$ will often be of the form $\xi \mid E$ where $\xi$ is a partition on $X$. The partition $\xi\mid E$ consists of elements $A \cap E$, $A \in \xi$. As we will often be dealing with maps from one measure space $f : (X,\mathcal{A},\mu)\to (Y,\mathcal{B},\nu)$, we note that when we say $f$ is \emph{measure-preserving}, we mean that for any $B\in\mathcal{B}$, we have $\mu(f^{-1}(B)) = \nu(B)$. Using these definitions, we introduce the notion of Very Weak Bernoulli partitions.

\begin{defn}
Let $(X,\mathcal{B},\mu)$ be a Lebesgue space, and let $f : X\to X$ be an invertible and  $\mu$-measure-preserving. Further, assume $X$ is a compact metric space and $\mathcal{B}$ is the completed Borel $\sigma$-algebra with respect to $\mu$, and let $\xi$ be a finite measurable partition. We say that $\xi$ is \textit{Very Weak Bernoulli} (VWB) if for any $\ep > 0$, there exists $N_0\in \N$ such that for all $m_2\geq m_1 \geq N_0$, for all $n > 0$, and for $\ep$-almost every $E\in\bigvee_{m_1}^{m_2}f^k\xi$, we have
$$\bar{d}( \{f^{-i}\xi\mid E\}_1^n, \{f^{-i}\xi\}_1^n) \leq \ep$$
with respect to $(E,\mu|_E)$ and $(X,\mu)$, where $f^k\xi = \{f^kA\mid A\in\xi\}$ for all $k\in\mathbb{Z}$.
\end{defn}

As the name suggests, a Bernoulli partition $\xi$ for $(X,f,\mu)$ is necessarily a Very Weak Bernoulli partition; see \cite[Theorem 3.1]{ponce2019introduction}. What makes Very Weak Bernoulli partitions  important is the converse implication developed by Ornstein outlined in the next two propositions: VWB partitions provide a method for establishing the Bernoulli property, which is often difficult via a direct approach. In fact, many results in the literature establishing the Bernoulli property follow this general outline \cite{pesin1977characteristic,ornstein1998bernoulli,chernov1996nonuniformly}.
\begin{prop} 
If a partition $\xi$ of $X$ is VWB, then writing $\mathcal{B}$ for the $\sigma$-algebra generated by $\bigvee_{-\infty}^{\infty}f^n\xi$, we have that $\left(X,f,\mathcal{B},\mu\right)$ is Bernoulli.
\end{prop}

\begin{prop} Suppose there exists an increasing sequence of finite measurable partitions $\{\xi_i\}_{i\in\N}$ such that $\displaystyle\bigvee_{i=1}^\infty\bigvee_{n=-\infty}^\infty f^{n}\xi_i$ generates the entire $\s$-algebra $\mathcal{B}$ and that $(X,f,\bigvee_{n=-\infty}^\infty f^{n}\xi_i,\mu)$ is Bernoulli for each $i\in \N$, then $(X,f,\mathcal{B},\mu)$ is Bernoulli.
\end{prop}

We now describe a mechanism for verifying that a partition is VWB.
We will show this by means of the following lemma, which involves constructing maps which, informally, can be thought of as
\begin{enumerate}
	\item Keeping \emph{most} points in the same set of partition elements
	\item Being \emph{mostly} measure-preserving
\end{enumerate}

To formalize these two concepts we need to introduce one more definition.

\begin{defn}\label{defn: almost mp}
Suppose $(X_i,\mathcal{A}_i,\mu_i)$ are probability spaces for $i = 1,2$. An $(\mathcal{A}_1,\mathcal{A}_2)$-measurable map $\theta : X_1\to X_2$ is \emph{$\eps$-measure-preserving} if there exists $G\in \mathcal{A}_1$ with $\mu_1(G) > 1 - \eps$ such that for any measurable $A\subset G$ with $\mu_1(A) > 0$, we have
$$\abs{\frac{\mu_2(\theta A)}{\mu_1(A)} - 1} < \eps.$$

\end{defn}

\begin{note}
Implicit in this definition is an assumption that for all $A\subset G$ of positive measure, $\theta A\in\mathcal{A}_2$. This assumption is immediately satisfied if $\theta$ is bimeasurable. More generally, it is satisfied if $\theta|_G : (G,\mathcal{A}_1\cap G)\to (\theta(G),\mathcal{A}_2\cap \theta(G))$ is bimeasurable. We will use both of these conditions in this paper.
\end{note}

Using this, and recalling Definition \ref{defn: eps almost everywhere} of ``$\ep$-almost everywhere'', we state the following lemma from \cite{chernov1996nonuniformly}.

\begin{lem}\label{lem: dbar closeness}
Let $\{\xi_i\}_{1}^n$ and $\{\eta_i\}_{1}^{n}$ be finite sequences of partitions of nonatomic Lebesgue probability spaces $(X,\mathcal{A},\mu)$ and $(Y,\mathcal{B},\nu)$ respectively. Let $\theta : X\to Y$ be $\ep$-measure-preserving, and suppose that there exists $G\subset X$ with $\mu(G) > 1-\ep$ so that for all $x\in G$,
$$h(x,\theta(x)) < \ep.$$
Then $\bar{d}(\{\xi_i\}_1^n,\{\eta_i\}_1^n) < C\ep$ for a constant $C$ independent of both the partitions and $\ep$.
\end{lem}

Applying this to show that a partition is Very Weak Bernoulli, we have the following corollary which we will use to establish Theorem \ref{thm: abstract}.
\begin{cor}\label{cor: VWB}
Let $(X,f,\mathcal{B},\mu)$ be a compact metric space equipped with the completion of the Borel $\sigma$-algebra with respect to a probability measure $\mu$, and let $\xi$ be a finite measurable partition. Let $f : X\to X$ an invertible measure-preserving transformation. Suppose that for all $\ep > 0$, there exists $N\in \N$ so that for any $m_2\geq m_1\geq N$, and for $\ep$-almost every $E\in\bigvee_{m_1}^{m_2} f^i\xi$, there exists a map $\theta : (E,\mathcal{B}\cap E,\mu|_E)\to (X,\mathcal{B},\mu)$ satisfying:
\begin{itemize}
	\item $\theta$ is $\ep$-measure-preserving;
	\item There exists $D\subset E$ with $\mu|_E(D) > 1-\ep$ so that for all $n > 0$ and $x\in D$, with respect to the partition sequences  $\{f^{-i}\xi \mid E\}_1^{n}$ and $\{f^{-i}\xi\}_1^n$, we have $$h(x,\theta(x)) < \ep.$$
\end{itemize}
Then $\xi$ is Very Weak Bernoulli. 
\end{cor}

\subsection{Useful Facts from Measure Theory}\label{subsec: measure theory}
We will make repeated use of the following results. Throughout, we will assume that $\ep < \frac{1}{5}$. As we will eventually need our results to hold only for $\ep$ arbitrarily close to $0$, we can freely make this assumption.

\begin{lem}\label{lem: large intersections}
Suppose $\mu$ is a probability measure on $X$ and $\mu(B) > 1-\ep$ and $A_1,\cdots, A_n$ are disjoint measurable sets. Then the union of all $A_i$ satisfying $\mu(A_i\cap B) < (1-\delta) \mu(A_i)$ has measure at most $\ep/\delta$.
\end{lem}

\begin{proof}
Let $\mathscr{A} = \{A_i\mid \mu(A_i\cap B) < (1-\delta)\mu(A_i)\}$. Now observe
\begin{align*}
	\mu\left(\bigcup_{A_i\in\mathscr{A}} A_i\right) &= \sum_{A_i\in\mathscr{A}} \mu(A_i\cap B) + \mu(A_i\cap B^c)
	\\
	&\leq (1-\delta)\sum_{A_i\in\mathscr{A}} \mu(A_i) + \mu(B^c)
	\\
	&\leq (1-\delta)\mu\left(\bigcup_{A_i\in\mathscr{A}} A_i\right) + \ep.
\end{align*}
Thus, $\displaystyle \mu\left(\bigcup_{A_i\in\mathscr{A}} A_i\right) < \frac{\ep}{\delta}$.
\end{proof}

We will often work with finite measures which we must later rescale to probability measures. This result simplifies our calculations in these cases.
\begin{prop}\label{prop: rescaling derivative}
Let $\ep > 0$ and $r\geq 2\ep + \ep^2$. Now suppose $(X,\mathcal{A},\mu)$ and $(X,\mathcal{B},\nu)$ are measure spaces and $A\subset X$ is $\mu$ and $\nu$ measurable with $\mu|_A \ll\nu|_A$. If $\abs{\frac{d\mu}{d\nu}\Big|_A - r} < \ep^2$, then $\abs{\frac{d\mu|_A}{d\nu|_A} - 1} < \ep$.
\end{prop}

\begin{proof}
Let $B\subset A$ be an arbitrary $\mu$ and $\nu$-measurable set. Then
$$\mu|_A(B) = \frac{\mu(B)}{\mu(A)} \leq \frac{\nu(B)(r+\ep^2)}{\nu(A)(r-\ep^2)} = \nu|_A(B)\frac{r+\ep^2}{r-\ep^2}.$$
Now, observe that 
$$\frac{r+\ep^2}{r-\ep^2} = 1 + \frac{2\ep^2}{r-\ep^2} \leq 1 + \ep$$
so long as $r\geq 2\ep + \ep^2.$ Similarly, we see that
$$\mu|_A(B) \geq \frac{r-\ep^2}{r+\ep^2}\nu|_A(B) = (1 - \frac{2\ep^2}{r+\ep^2})\nu|_A(B).$$
As
$\displaystyle \frac{2\ep^2}{r+\ep^2}\leq \ep,$
we have the reverse inequality.
\end{proof}

The following proposition shows that if we have a measure-preserving function on most of the space, then we can change the measures slightly and the function will still be $\ep$-measure-preserving.

\begin{prop}\label{prop: perturbation of measure preserving}
Let $\ep>0$ be arbitrary. Suppose $(X_1,\mathcal{A}_1,\nu_1),(X_2,\mathcal{A}_2,\nu_2)$ are probability spaces, and suppose $\mu_1,\mu_2$ are also probability measures on $(X_1,\mathcal{A}_1)$ and $(X_2,\mathcal{A}_2)$ respectively. Consider $Y_1\in\mathcal{A}_1$ and $Y_2\in \mathcal{A}_2$ satisfying the following conditions:
	\begin{enumerate}
		\item On $Y_1$, $\abs{\frac{d\mu_1}{d\nu_1} - 1} < \ep$;
		\item On $Y_2$, $\abs{\frac{d\mu_2}{d\nu_2} - 1} < \ep$;
		\item $\mu_1(Y_1)\geq 1-\ep$ and $\mu_2(Y_2) \geq 1-\ep$.
	\end{enumerate} 
	Then, if $\theta : X_1\to X_2$ is injective (up to sets of $\nu_1$-measure $0$), bimeasurable, and measure-preserving as a map from $(X_1,\mathcal{A}_1,\nu_1)$ to $(X_2,\mathcal{A}_2,\nu_2)$, then 
	it is $5\ep$-measure-preserving as a map from $(X_1,\mathcal{A}_1,\mu_1)$ to $(X_2,\mathcal{A}_2,\mu_2)$. 
\end{prop}

\begin{proof}
Let $A\subset Y_1\cap \theta^{-1}Y_2$ be in $\mathcal{A}_1$. Then observe
\begin{align*}
\mu_2(\theta A) &= \int_{\theta A}\frac{d\mu_2}{d\nu_2}\,d\nu_2
\\
&\leq (1+\ep)\nu_2(\theta A)
\\
&= (1+\ep)\nu_1(A)
\\
&= (1+\ep)\int_A \frac{d\nu_1}{d\mu_1}\,d\mu_1
\\
&\leq \frac{1+\ep}{1-\ep}\mu_1(A)
\\
&\leq (1 + 5\ep)\mu_1(A).
\end{align*}
The reverse inequality follows similarly.

Now, we will now show that $\mu_1(Y_1\cap \theta^{-1}(Y_2)) \geq 1 - 5\ep$. First, observe that 
$$\nu_1(Y_1) \geq \frac{1}{1+\ep}\mu_1(Y_1) \geq \frac{1-\ep}{1+\ep} \geq 1 - 2\ep$$
with the same proof showing $\nu_2(Y_2) \geq 1 - 2\ep$.
Observe that
\begin{align*}
\mu_1(Y_1\cap \theta^{-1}(Y_2)) &\geq (1-\ep)\nu_1(Y_1\cap \theta^{-1}(Y_2))
\\
&\geq (1-\ep)(\nu_1(\theta^{-1}(Y_2)) -2\ep)
\\
&= (1-\ep)(\nu_2(Y_2) - 2\ep)
\\
&\geq (1-\ep)(1 - 4\ep)
\\
&\geq 1 - 5\ep
\end{align*}
where the second inequality follows because $\nu_1(Y_1^c) \leq 2\ep$. 
\end{proof}

\subsection{Creating Comparison Sets}

The following proposition essentially says that if we have two equivalent measures, we can divide most of our space up into ``nice'' subsets, on which the ratio of our two measures does not change much. Eventually, we will apply this result using dynamical rectangles to formalize the argument that we need only measurable local product structure to apply the classical argument that $K$ implies Bernoulli.

\begin{prop}\label{prop: creation of general covering}
	Suppose that $\mu$ and $\lambda$ are Borel probability measures on a compact metric space $X$ with $\mu\ll\lambda$, and $\mathscr{C}$ is a generating semi-ring for the Borel $\sigma$-algebra $\mathcal{B}$. Then for all $\ep_0 > 0$, there exists a pairwise disjoint collection $\{R_i\}\subset\mathscr{C}$ equipped with the measure $\lambda|_{R_i}$ satisfying the following:
	\begin{enumerate}
		\item There exists $G_i\subset R_i$ such that $\mu|_{R_i}(G_i) > 1 - \ep_0$ and $\lambda|_{R_i}(G_i) > 1 -\ep_0$;
		\item $\mu(\bigcup_i G_i) \geq 1-\ep_0$.
		\item For each $R_i$, $\abs{\frac{d\mu|_{R_i}}{d\lambda|_{R_i}}|_{G_i} - 1} < \ep_0$;
	\end{enumerate}
\end{prop}

\begin{proof}
Let $\ep=\ep_0/4$ which will be the scale at which we will carry out the construction below. Now, because $\mu \ll\lambda$, we can choose a representative measurable function of $\frac{d\mu}{d\lambda}$ which is non-zero and finite on a set $Y_0$ with $\mu(Y_0) = 1$. By regularity of $\mu$ and $\lambda$, take $Y_1\subset Y_0$ to be compact, such that $\mu(Y_1) > 1-\ep^4$ and $\lambda(Y_1) > (1-\ep^4)\lambda(Y_0)$. Then, by applying Lusin's theorem, there exists a compact set $K\subset Y_1$ on which $\frac{d\mu}{d\lambda}\Big|_K$ is continuous, and $$\mu(K) > 1 - \ep^3 \text{ and }\lambda(K) > (1-\ep^3).$$
By continuity of the derivative, compactness of $K$, and the fact that the derivative is non-zero on $K$, we have definite upper and lower bounds
$$0 < m \leq \frac{d\mu}{d\lambda}\Big|_K \leq M < \infty.$$
Furthermore, we take $\gamma < \ep$ small enough so that $m\geq 2\gamma + \gamma^2$ in order to apply Proposition \ref{prop: rescaling derivative}. Then, by continuity of the derivative on $K$, we can choose $\delta > 0$ such that on any set $R$ with diameter at most $\delta$, there exists $r \in [m,M]$ such that $\abs{\frac{d\mu}{d\lambda}\Big|_{K\cap R} - r} \leq \gamma^2$. In particular, Proposition \ref{prop: rescaling derivative} and the choice of $\gamma$ gives $$\abs{\frac{d\mu|_{K\cap R}}{d\lambda|_{K\cap R}} - 1} < \ep$$ for such $R$.

We now prove the first and second conditions of the proposition in the following lemma.
\begin{lem}
Let $\mu,\lambda$ and $\mathscr{C}$ be as in Proposition \ref{prop: creation of general covering}, and $\delta$ and $K$ be as chosen above. Then there exists a finite collection $\{R_i\}_{i=1}^n \subset \mathscr{C}$ with $\text{diam}(R_i) < \delta$ and subsets $G_i \subseteq R_i$ such that
$$\mu|_{R_i}(G_i) \geq 1 -\ep \text{ and } \lambda|_{R_i}(G_i) \geq 1-\ep$$
and $\mu\left(\bigcup_{i=1}^n G_i\right) \geq 1-2\ep$.
\end{lem}
	\begin{proof}

	Using the fact that $\mu\ll\lambda$, fix $0 < \beta < \ep$ so that for any measurable $A\subset X$ with $\lambda(A) < \beta^3$ satisfies $\mu(A) < \ep^3$. Since $\mathscr{C}$ is a generating semi-ring, there exists a finite collection of pairwise disjoint sets $\mathscr{R}\subset\mathscr{C}$ of diameter at most $\delta$ for which we can approximate $K$ arbitrarily well with respect to $\lambda$; that is, we choose $\mathscr{R}$ so that
\begin{equation}\label{eq: scribed R}
\lambda\left(K\vartriangle \bigcup_{R\in\mathscr{R}}R\right) < \beta^4. 
\end{equation}

	This implies that the union of $R\in \mathscr{R}$ with $\lambda(R\cap K) < (1-\ep)\lambda(R)$ has very small measure with respect to $\lambda$. In particular, letting $\mathscr{B}_{\lambda}$ denote the collection of such $R$, the analogous argument of Lemma \ref{lem: large intersections} applied using \eqref{eq: scribed R} gives

	\begin{align*}
		\lambda\left(\bigcup_{\mathscr{B}_{\lambda}}R\right) &\leq \frac{1}{\ep}\sum_{\mathscr{B}_{\lambda}}\lambda(R\cap K^c)
		\\
		&\leq \beta^3.
	\end{align*}
	Therefore, $\displaystyle \mu\Big(\bigcup_{\mathscr{B}_\lambda}R\Big) < \ep^3$ from the choice of $\beta$.
	
Furthermore, again by our choice of $\beta$ and \eqref{eq: scribed R}, we know that $\displaystyle \mu\left(K\vartriangle \bigcup_{R\in\mathscr{R}}R\right) < \ep^3$, and so, defining $\mathscr{B}_\mu$ analogously as the set of $R\in \mathscr{R}$ with $\mu(R\cap K) < (1-\ep)\mu(R)$, we have that
$$\mu\left(\bigcup_{\mathscr{B}_{\mu}}R\right) \leq \frac{1}{\ep}\sum_{\mathscr{B}_{\mu}}\mu(R\cap K^c)\leq \ep^2.$$

Consequently,
$$\displaystyle \mu\left(\bigcup_{\mathscr{R}\setminus (\mathscr{B}_\lambda\cup \mathscr{B}_\mu)} R\right) > \mu(K) - \ep^3 - \ep^3 - \ep^2 > 1 - 3\ep^3 - \ep^2 > 1 - 2\ep^2.$$
Henceforth, we take our collection $\{R_i\}_{i=1}^n$ to be $\mathscr{R}\setminus (\mathscr{B}_\mu \cup \mathscr{B}_{\lambda})$, and we take
$$G_i := R_i\cap K.$$ 
In other words, our collection $\{R_i\}_{i=1}^n$ consists of $R\in \mathscr{R}$ whose intersection with $K$ has large relative measure with respect to both $\mu$ and $\lambda$.
The first two inequalities in the lemma follow from the fact that $R_i$ belongs to neither $\mathscr{B}_\mu$ nor $\mathscr{B}_\lambda$.

All that remains is to show that $\mu(\bigcup_{i=1}^n G_i) \geq 1-\ep$. It suffices to show that $\mu(\bigcup_{i=1}^n R_i) \geq 1-\ep$, as this implies
$$\mu\left(\bigcup_{i=1}^n G_i\right) = \sum_{i=1}^n \mu(G_i) \geq (1-\ep) \sum_{i=1}^n \mu(R_i) \geq (1-\ep)^2 \geq 1-2\ep.$$
Now observe that
	$$\mu\left(K\setminus \bigcup_{i=1}^n R_i\right) \leq \mu\left(K\setminus \bigcup_{\mathscr{R}}R\right) + 2\ep^2 \leq \ep^3 + 2\ep^2 \leq 3\ep^2.$$
Therefore, since $\mu(K) > 1 - \ep^3$, we have that
	$$\mu\left(\bigcup_{i=1}^n R_i\right) \geq \mu\left(K\cap \bigcup_{i=1}^n R_i\right) = \mu(K) - \mu\left(K\setminus \bigcup_{i=1}^n R_i\right) > 1 - \ep^3 - 3\ep^2 > 1 - \ep.$$
\end{proof}

To complete our proof of the proposition, we need to show the third condition, that for each $R_i$, $\abs{\frac{d\mu|_{R_i}}{d\lambda|_{R_i}}|_{G_i} - 1} < \ep_0$. Informally, this condition says that for any measurable subset of $G_i$, the measures with respect to $\mu|_{R_i}$ and $\lambda|_{R_i}$ are approximately the same. Now, recall $G_i$ has relatively large $\mu$ and $\lambda$ measure in $R_i$, so we have
	$$1-\ep < \frac{d\mu|_{R_i}}{d\mu|_{G_i}} \leq 1 \quad \text{ and }\quad 1 \leq \frac{d\lambda|_{G_i}}{d\lambda|_{R_i}} < \frac{1}{1-\ep} \leq 1 + 2\ep.$$
	Additionally, recall from the construction that elements of $\mathscr{R}$ have small enough diameter (specifically less than $\delta$) so that $\abs{\frac{d\mu|_{G_i}}{d\lambda|_{G_i}} - 1}<\ep.$ Therefore, since for any measurable $B\subset G_i\subset R_i$, we have
	$$\frac{\mu|_{R_i}(B)}{\lambda|_{R_i}(B)} = \frac{\mu(B)\lambda(R_i)}{\mu(R_i)\lambda(B)} = \left(\frac{\mu|_{G_i}(B)}{\lambda|_{G_i}(B)}\right)\left(\frac{\mu(G_i)}{\mu(R_i)}\right)\left(\frac{\lambda(R_i)}{\lambda(G_i)}\right),$$
	it follows that $\abs{\frac{d\mu|_{R_i}}{d\lambda|_{R_i}} - 1} < 4\ep =\ep_0$.
\end{proof}

\begin{cor}\label{cor: general covering}
Let $X$ be a compact metric space, $\mu$ a Borel probability measure, and suppose for $\mu$-a.e. $x\in X$, there exists an open, positive $\mu$-measure set $U_x\ni x$ such that $\mu|_{U_x} \ll \lambda_x$ for some probability measure $\lambda_x$. Let $\mathscr{C}$ be a generating semi-ring for the Borel $\sigma$-algebra $\mathcal{B}$. Then for all $\ep_0 > 0$, there exists a pairwise disjoint collection $\{R_i\}\subset\mathscr{C}$ satisfying the same conditions of Proposition \ref{prop: creation of general covering}, with each $R_i$ equipped with a measure $\lambda_i|_{R_i}$.
\end{cor}

\begin{proof}
Let $A$ denote the full measure set of $x\in X$ for which there exists an open set $U_x$ as described in the statement. Then, for all $\delta > 0$, define $B_\delta\supset X\setminus A$ to be an open set with $\mu(B_\delta) < \delta$. Then $\{U_x\}_{x\in A}\cup B_{\ep_0/2}$ is an open cover of $X$, and so there exists $x_1,\cdots, x_n$ such that $X = B_{\ep_0/2} \cup \bigcup_{i=1}^n U_{x_i}$. Consequently, we see that $\mu(\bigcup_{i=1}^n U_{x_i}) > 1 - \frac{\ep_0}{4}$. Now, for each $i$, we  consider a compact subset of $U_{x_i}\setminus \bigcup_{j = 1}^{i-1} U_{x_j}$ denoted $K_i$, with $\mu(K_i) > (1-\frac{\ep_0}{4})\mu(U_{x_i}\setminus \bigcup_{j=1}^{i-1} U_{x_j})$. 

We then apply Proposition \ref{prop: creation of general covering} to the compact metric space $K_i$ with probability measures $\mu|_{K_i}$ and $\lambda_{x_i}|_{K_i}$ and the generating semi-ring $\mathscr{C} \cap K_i$ for constant $\frac{\ep_0}{4}$. Now, the final step is to ensure that our collection of elements of $\mathscr{C}$ is pairwise disjoint, as a priori, one could have the same $R\in\mathscr{C}$ which overlaps both $K_i$ and $K_j$. So, we work only with $R\in\mathscr{C}$ with diameter less than the minimum distance between any two of our compact sets $K_i,K_j$. This is still a generating semiring, and there are finitely many such compact sets, so our proof is complete.
\end{proof}

\begin{rem}
Proposition \ref{prop: creation of general covering} can almost be interpreted as saying that there is an $\ep$-regular covering as defined by Chernov and Haskell \cite{chernov1996nonuniformly}. The only missing component is that we do not impose extra assumptions such as complete hyperbolicity on the map, and as such, we do not make any claims on the structure of $\mathscr{C}$. 
Once we begin restrict our setting to be a Smale space, the generating semi-ring $\mathscr{C}$ will be the set of rectangles (defined in the next section).
\end{rem}

\section{Proof of Theorem \ref{thm: abstract}}

From here on out, we restrict ourselves to a Smale space $(X,f)$ with hyperbolicity constant $\chi$ and constant $\eps_X$ (see Section \ref{subsec:metric Anosov} for details about this setting). We also equip our space with a complete, invariant, Borel probability measure $\mu$ with measurable local product structure, with $\mathcal{B}$ denoting the completion of the Borel $\sigma$-algebra. Thus, $(X,\mathcal{B},f,\mu)$ is a Lebesgue space. We will also work with a finite, measurable partition $\xi$ with the property that there exists $C_{\xi} > 0$ such that for all $\ep > 0$, 
\begin{equation}\label{eq: boundary measure}
\mu(B(\partial \xi, \ep)) < C_\xi\ep,
\end{equation}
where $B(\partial \xi,\ep):=\{x \in X \colon d(x,\partial A) \leq \eps \text{ for some } A \in \xi\}$.
These partitions can be found with arbitrarily small diameter in any compact metric space. This result is contained in the remark following \cite[Lemma 4.1]{ornstein1998bernoulli}, but in the interest of self-containment, we provide the proof here. The key fact underlying it is the following lemma.

\begin{lem}[\cite{ornstein1998bernoulli}]
If $(X,d)$ is a compact metric space and $\mu$ is a probability measure on $X$, then given $x_0\in X$ and an interval $[a,b]\subset (0,\infty)$ with $a < b$, there exists $r\in [a,b]$ such that for all $\eps > 0$,
$$\mu\left(\left\{x\mid \abs{d(x,x_0)  - r} \leq \eps\right\}\right) \leq \frac{6}{b-a}\eps.$$
\end{lem}

Using this, we get the desired corollary.

\begin{cor}
If $(X,d)$ is a compact metric space and $\mu$ is a probability measure on $X$, then for all $\delta > 0$, we can construct a finite Borel measurable partition $\xi$ with the property that there exists $C_\xi > 0$ such that for all $\eps > 0$,
$$\mu(B(\partial \xi,\eps)) < C_\xi\eps.$$
\end{cor}

\begin{proof}
Let $4\delta > 0$ be our desired maximal diameter. Then, in the previous lemma, define $r_x > 0$ to be the number given for the point $x\in X$ and the interval $[\delta,2\delta]$. Then, as $\{B(x,r_x)\}_{x\in X}$ is an open cover of $X$, by compactness, we can take $\{A_i\}_{i=1}^n$ to be a finite subcover. Now define the partition $\xi = \{U_1,U_2,\cdots, U_n\}$ by taking $U_1 = A_1$, and $U_i = A_i\setminus \bigcup_{j=1}^{n-1}U_j.$. This is a Borel partition by construction, and each element has diameter at most $2r_{x_i} \leq 4\delta$. So, we need only check the boundary condition. Let $\eps > 0$ and $1\leq i\leq n$ be arbitrary. Suppose that $d(x,\partial U_i) \leq \eps$. Then $d(x,\partial A_j)) \leq \eps$ for some $1\leq j\leq n$. Therefore,
$$\mu(B(\partial \xi),\eps) \leq \mu\left(\bigcup_{i=1}^n B(\partial A_i,\eps)\right) \leq \sum_{i=1}^n\mu\left(\left\{x \mid \abs{d(x,x_i) - r} \leq \eps\right\}\right) \leq \sum_{i=1}^n \frac{6}{\delta}\eps. = \frac{6n}{\delta}{\eps}.$$
Taking $C_\xi = \frac{6n}{\delta}$ completes our proof.
\end{proof}

\subsection{Rectangles} 
We now introduce the notion of rectangles, which will serve as the building blocks when we construct the map $\theta$ from Corollary \ref{cor: VWB}.

\begin{defn}
Let $z\in X$, and consider Borel $A\subset X^s(z,\delta_1)$ and Borel $B\subset X^u(z,\delta_2)$ for some $\delta_1,\delta_2 > 0$. If these are chosen small enough so that $[A,B]$ is well-defined, we call $[A,B]$ a \textit{rectangle}.
\end{defn}

In particular, we can observe from this definition that a rectangle $R$ is the image of a Borel set under a homeomorphism, and thus is Borel. We can also observe from this definition that for a given rectangle $R$, $[R,R] = R$.

We also can observe the following result without difficulty.
\begin{prop}
Rectangles can be constructed with arbitrarily small diameter, and in particular, the set of rectangles is a generating semiring for $\mathcal{B}$.
\end{prop}

We will eventually need the notion of a layerwise intersection in order to define $\theta$ from Corollary \ref{cor: VWB}.
\begin{defn}
Given a rectangle $R$, then $A\subset R$ intersects $R$ \textit{layerwise} if for all $x\in A$, the unstable set of $x$ at scale $\eps_X$ stretches across the entirety of $R$, i.e.
	$$X^u(x,\eps_X)\cap R \subset X^u(x,\eps_X)\cap A.$$
\end{defn}

As it turns out, given our partition $\xi$, and finitely many disjoint rectangles, ``most'' $E\in \bigvee_{m_1}^{m_2}f^j\xi$ have a subset of large relative measure which intersects each rectangle layerwise, provided that $m_1,m_2$ are large enough.

\begin{prop}\label{prop: layerwise} Let $\mathcal{R} = \{R_i\}_{1}^n$ be a finite set of disjoint rectangles in $X$.
For all $0 < \ep < \eps_X$, there exists $N \in \N$ such that for all $m_2 \geq m_1 \geq N$, for $\ep$-almost every $E\in \bigvee_{m_1}^{m_2} f^j\xi$, there exists Borel $A\subset E$ which intersects each $R_i$ layerwise and satisfies $\mu(A) \geq (1-\ep)\mu(E)$.
\end{prop}

\begin{proof}
For $ x\in \bigcup_{R\in\mathcal{R}}R$, let $R(x) \in \mathcal{R}$ be the rectangle which contains $x$. For any $F \in \xi$ and $k\in \N$, we denote the subset of $f^kF$ consisting of points that do not intersect the rectangles in $\mathcal{R}$ layerwise as
$$A^k_F := \left\{x\in f^k F \cap \bigcup_{R\in\mathcal{R}}R \bigm| X^u(x,\ep)\cap R(x) \not\subset f^k F\cap R(x)\right\}.$$

This can equivalently be written as $\bigcup_{R\in\mathcal{R}}f^kF\cap [R,R\setminus f^kF]$, as $x\in [R(x),R(x)\setminus f^kF]$ if and only if $X^u(x,\ep)$ has non-empty intersection with the $R(x)\setminus f^kF$. This makes it clear that $A^k_F$ is Borel, as it is the intersection of Borel sets.

Observe that as rectangles have diameter at most $\eps_X$, then if $x\in  f^{-k}A^k_F$, we have $X^u(f^kx,\eps_X)\cap \partial F\neq \emptyset$. Consequently, $d(x,\partial F) < \eps_X\chi^{k}$, where $\chi$ is the hyperbolicity constant for $f$. Therefore, $\mu(A^k_F) < (C_\xi\eps_X)\chi^k$

Hence, defining $\displaystyle B_k = \bigcup_{F\in \xi}A^k_F $ to be the union of all such ``bad'' parts of our atoms, by our choice of partition \eqref{eq: boundary measure}, it follows that $\mu(B_k) \leq (C_{\xi}\eps_X)\chi^{k}$. Now choose $N \in \N$ large enough so that $\displaystyle B := \bigcup_{k=N}^{\infty} B_k$ satisfies $\mu(B) < \ep^2$.

Let $m_2 \geq m_1 \geq N$ be arbitrary. As $\mu(B) < \ep^2$, by Lemma \ref{lem: large intersections}, for $\ep$-almost every $E\in \bigvee_{m_1}^{m_2}f^i\xi$, we have $\mu(E\cap B^c) \geq (1-\ep)\mu(E)$. For such $E$, we set $A = E\cap B^c$ which intersects each $R_i$ layerwise, as required.
\end{proof}

The next proposition is the main result of this section. Its proof, which appears in Section \ref{sec: final}, involves several intermediate steps making use of results established in Section \ref{sec: Bernoulli}. Hence, we postpone it to the next proposition for better readability.

\begin{prop}\label{prop: pre most E}
For any $\ep>0$, there exists $N\in \N$ and $C \geq 2$ such that for all $m_2 \geq m_1 \geq N$, there exists $\mathscr{B} \subseteq \bigvee_{m_1}^{m_2} f^i\xi$ with $\mu(\bigcup\limits_{E \in \mathscr{B}} E) > 1-\ep$ such that for every $E \in \mathscr{B}$ there exists a Borel subset $A\subset E$ and a $C\ep$-measure-preserving map $\theta : (A,\mathcal{B}\cap A,\mu|_A)\to (X,\mathcal{B},\mu)$ satisfying $$\mu|_E(A) \geq (1-\ep) \text{ and }\theta(x) \in X^s(x,\eps_X) \text{ for all }x\in A.$$ 
Moreover, given any $\delta>0$, the construction can be performed so that $d(x,\theta(x)) <\delta$ for every $x\in A$.
\end{prop}

The following corollary extends the domain of $\theta$ introduced in the above proposition from $A$ to $E$ while maintaining its properties.

\begin{cor}\label{cor: culmination}
The map $\theta \colon A \to X$ from the above proposition can be extended to a $C\ep$-measure-preserving map from $(E,\mathcal{B}\cap E,\mu|_E)$ to $(X,\mathcal{B},\mu)$ while maintaining $\theta(x) \in X^s(x,\eps_X)$ for every $x\in E$.
\end{cor}
\begin{proof}
We trivially extend $\theta : (E,\mu|_E)\to (X,\mu)$ by setting $\theta(x) = x$ for $x\notin A$. We need to check that the map is $\ep$-measure-preserving.

First, we have that $(E,\mathcal{B}\cap E,\mu|_E)$ is a Lebesgue space, as $E\subset X$ is measurable and $\mu|_E$ is the conditional measure. The remainder of the proof uses only the fact that the original map $\theta : (A,\mathcal{B}\cap A,\mu|_A)\to (X,\mathcal{B},\mu)$ is $C\ep$-measure preserving. Denoting by $G$ the subset of $A$ on which $\theta$ is $C\ep$-measure preserving (see Definition \ref{defn: almost mp}), the same set $G$ can be used for verifying the $\ep$-measure preserving property for the extended map $\theta : (E,\mathcal{B}\cap E,\mu|_E)\to (X,\mathcal{B},\mu)$.
Indeed, from its definition we have
$$\left|\frac{\mu(\theta D)}{\mu(D)} - 1\right| < C\ep$$
for any positive measure $D\subset G$ (noting that if $\mu|_A(D) > 0$, then $\mu|_E(D) > 0$, as $A\subset E$).
Moreover, the set $G$ still has large $\mu$-measure in $E$ because $A$ does:
$$\mu(B) \geq (1-\ep)\mu(A) \geq (1-\ep)^2\mu(E) \geq (1-2\ep)\mu(E).$$
Consequently, we see that $\theta : (E,\mathcal{B}\cap E,\mu|_E)\to (X,\mathcal{B},\mu)$ is $C\ep$-measure-preserving.
\end{proof}

Assuming Proposition \ref{prop: pre most E}, we can complete the proof of Theorem \ref{thm: abstract}.

\begin{prop}
The partition $\xi$ satisfying \eqref{eq: boundary measure} is Very Weak Bernoulli.
\end{prop}

\begin{proof}
We will use Corollary \ref{cor: VWB} to verify that $\xi$ is Very Weak Bernoulli. Let $\ep > 0$ be arbitrary. Choose $\delta \in (0, \ep^2/C_{\xi})$ where $C_{\xi}$ is the constant from \eqref{eq: boundary measure}; such choice of $\delta$ implies $\mu(B(\partial \xi,\delta)) \leq \ep^2$. Let $N \in \N$ be the integer obtained from applying Proposition \ref{prop: pre most E} to $\ep$ and $\delta$. 

For any $m_2 \geq m_1\geq N$, let $\mathscr{B}$ be the elements of $\bigvee_{m_1}^{m_2} f^i\xi$ from Corollary \ref{cor: culmination} with $\mu(\bigcup\limits_{E \in \mathscr{B}} E) > 1-\eps$. 
For any such $E \in \mathscr{B}$ and any $n\in \N$, consider two sequences of partitions: $\{f^{-i}\xi \mid E\}_1^n$ of $E$ and $\{f^{-i}\xi\}_1^n$ of $X$. If
$$h(x,\theta(x)) > \ep$$
for some $x\in E$, using the fact that $h(x,\theta(x))\neq 0$, there exists $1 \leq m \leq n$ so that $$f^{-m}\xi|_E (x)\neq f^{-m}\xi(\theta(x));$$ that is, $x$ and $\theta(x)$ belong to different partition elements of $f^{-m}\xi$. Since $\theta(x) \in X^s(x,\eps_X)$ and $d(x,\theta(x)) < \delta$, the distance between $f^{-m}\partial \xi$ and $x$ must also be at most $\delta$. Furthermore, under forward iterations of $f$, this distance will only decrease. In particular, we have $f^mx\in B(\partial \xi,\delta)$. Setting
$$B_E := \{x \in E \mid h(x,\theta(x)) \geq  \ep \}\subset B(\partial \xi,\delta).$$
and $B := \bigcup\limits_{E \in \mathscr{B}} B_E$, we then have $\mu(B) \leq \ep^2$ from the choice of $\delta$.

By applying Lemma \ref{lem: large intersections}, which amounts to excluding further those elements $E$ from $\mathscr{B}$ which satisfy $\mu(E \cap B^c) <(1-\ep)\mu(E)$, whose union has $\mu$-measure at most $\ep$, we obtain $\mathscr{C} \subseteq \mathscr{B}$ with $\mu\left(\bigcup\limits_{E \in \mathscr{C}} E\right) > 1 - 2\eps$ such that for every $E\in\mathscr{C}$ the subset $D:=E\cap B^c$ satisfies $$\mu|_E(D) > 1-\ep \text{ and }h(x,\theta(x)) < \ep \text{ for }x\in D.$$
 Therefore, $\xi$ is Very Weak Bernoulli from Corollary \ref{cor: VWB}.
\end{proof}

\begin{cor}
The system $(X,f,\mu)$ satisfying the assumptions of Theorem \ref{thm: abstract} is Bernoulli.
\end{cor}

\begin{proof}
We can find a generating sequence of Very Weak Bernoulli partitions, because we can choose a partition $\xi$ with an arbitrarily small diameter; see \cite[Lemma 4.1]{ornstein1998bernoulli}. 
Hence, from the above proposition the system $(X,f,\mu)$ is Bernoulli as required.
\end{proof}

\section{Proof of Proposition \ref{prop: pre most E}}\label{sec: final}
In this section we prove Proposition \ref{prop: pre most E}. We take $(X,\mathcal{B},f)$ to be a Smale space equipped with a complete, Borel, invariant measure $\mu$ which has measurable local product structure and has the $K$-property. We also fix constants $\ep \in (0,\frac{1}{10})$ and $\delta > 0$.

Now we recall what we wish to show. We need to establish the existence of $N\in\N$ and $C \geq 2$ so that for all $m_2\geq m_1\geq N$, there exists a collection $\mathscr{B}$ of partition elements of $\bigvee_{j=m_1}^{m_2} f^j\xi$ so that for all $E\in\mathscr{B}$:
\begin{itemize}
    \item There exists Borel $A\subset E$ with $\mu|_E(A)\geq (1-\ep)$
    \item There exists a $C\ep$-measure-preserving map $\theta : (A,\mathcal{B}\cap A,\mu|_A)\to (X,\mathcal{B},\mu)$ with $\theta(x)\in X^s(x,\eps_X)$ for all $x\in A$
    \item $d(x,\theta(x)) < \delta$ for all $x\in A$.
    \item $\mu\left(\bigcup_{E\in\mathscr{B}}E\right) > 1 - \ep.$
\end{itemize}

We provide a brief overview before elaborating the details. Note that we already have a candidate for the subset $A \subset E$ from Proposition \ref{prop: layerwise}. However, it is difficult to establish the $C\ep$-measure-preserving map $\theta$ directly, so we introduce an auxiliary measure $\lambda_0$ built from the local product measure on $X$ which facilitates the comparison between $\mu|_A$ and $\mu$. This comparison is made possible via well-chosen large subsets of $A$ and $X$ with certain properties.   
In the next subsection, we describe how we choose these large subsets to facilitate the comparison, and in the subsequent subsection we describe the construction of $\theta$ with the desired properties.

\subsection{Choice of a suitable subset of \texorpdfstring{$A$}{A}}
Because $\mu$ has measurable local product structure, and rectangles are a generating semi-ring, we begin by applying Corollary \ref{cor: general covering} at scale $\ep^2$ to obtain a finite collection of rectangles $\{R_i\}_{i=1}^n$ of diameter at most $\frac{\delta}{2}$ equipped with non-atomic Borel product measures $\lambda_i$ and designated subsets $G_i\subset R_i$. These satisfy the following properties: 
\begin{enumerate}
    \item\label{1} $\mu|_{R_i}(G_i) > 1-\ep^2$ and $\lambda_i(G_i) > 1-\ep^2$; 
	\item\label{2} On $G_i$, $\abs{\frac{d\mu|_{R_i}}{d\lambda_i} - 1} < \ep^2$;
	\item\label{eq: mu(G)} $\mu(G) \geq 1-\ep^2$ where $G := \bigcup\limits_{i=1}^n G_i$.
\end{enumerate}
We stress that the scale chosen to apply Corollary \ref{cor: general covering} is $\ep^2$; the choice of this constant will be used in Corollary \ref{cor: mu and lambda0}. The sets $G_i$ chosen with these properties play an important role in that the large subset of $A$ mentioned above will be given by $A \cap \bigcup G_i$.

\begin{lem}\label{lem: A subset E}
There exists $N\in \N$ such that for all $m_2 \geq m_1 \geq N$, there exists $\mathscr{B}_1 \subset \bigvee_{j = m_1}^{m_2} f^j\xi$ with $\mu\left(\bigcup\limits_{E\in\mathscr{B}_1}E\right) \geq 1-\ep/2$ such that for every $E\in\mathscr{B}_1$, there exists $A\subset E$ satisfying 
\begin{enumerate}[label=(\alph*)]
	\item $A$ intersects each $R_i$ layerwise;
	\item $\mu(A) \geq (1-\ep)\mu(E)$;
	\item $\mu(A\cap G)\geq (1-\ep)\mu(A)$.
\end{enumerate}
\end{lem}
\begin{proof}
We begin by applying Proposition \ref{prop: layerwise} at scale $\frac{\ep}{4}$ to the collection of rectangles $\{R_i\}_{i=1}^n$. This gives us $N\in \N$ such that for all $m_2 \geq m_1 \geq N$, for $\frac{\ep}{4}$-almost every $E\in \bigvee_{m_1}^{m_2} f^i\xi$, there exists Borel $A\subset E$ which intersects each $R_i$ layerwise and satisfies 
\begin{equation}\label{eq: A and E}
\mu(A) \geq \Big(1-\frac{\ep}{4}\Big)\mu(E).
\end{equation}
By assumption, $\mu$ has the $K$-property, and so applying Proposition \ref{prop: K-partitions} we may assume $N$ is chosen large enough so that for all $m_2 \geq m_1 \geq N$, for $\frac{\ep}{4n}$-almost every $E\in \bigvee_{j=m_1}^{m_2} f^j\xi$ and for $1\leq i\leq n$, we have
\begin{equation}\label{eq: K apply}
\left|\frac{\mu(E\cap G_i)}{\mu(E)} - \mu(G_i)\right| \leq \frac{\ep}{4n}.
\end{equation} 
This establishes estimates on the intersection of $\frac{\ep}{4n}$-almost every partition element with \textit{every} $G_i\subset R_i$, for $1\leq i \leq n$.

With our choice of $N$ so defined, fix $\mathscr{B}_1$ to be the set of elements $E \in \bigvee_{m_1}^{m_2} f^j\xi$ that satisfy both \eqref{eq: A and E} and \eqref{eq: K apply}. Consequently, for every $E\in\mathscr{B}_1$, the layerwise subset $A\subset E$ guaranteed by \eqref{eq: A and E} satisfies conditions (a) and (b) of the lemma.

We now show the third condition (c). Observe that
$$\mu\left(\bigcup\limits_{E\in\mathscr{B}_1}E\right) \geq 1- \frac{\ep}{4} - \frac{\ep}{4n} \geq 1 - \frac{\ep}{2}.$$
Then for every $E \in \mathscr{B}_1$, summing $\mu(E \cap G_i)$ over all $i$ using \eqref{eq: K apply} gives 
\begin{equation}\label{eq: E n G}
\mu(E\cap G) = \sum\limits_{i=1}^n \mu(E \cap G_i) \geq \mu(E)\sum_{i=1}^{n}\mu(G_i) - \frac{\ep}{4n} \geq  \mu(E)\Big(1-\frac{3\ep}{4}\Big)
\end{equation}
where the last inequality uses the fact that $\mu(G) \geq 1-\ep^2$.
As $\mu(A)\geq (1-\frac{\ep}{4})\mu(E)$, it now follows that
$$\mu(A\cap G) \geq \mu(E\cap G) - \mu(E\setminus A) \geq (1-\ep)\mu(E) \geq (1-\ep)\mu(A),$$
completing the proof.
\end{proof}

This lemma verifies one of the conditions in Proposition \ref{prop: pre most E} for all $E\in\mathscr{B}_1$, namely that
$$\mu|_E(A) \geq 1-\ep.$$
However, we still need to show the existence of a $C\ep$-measure-preserving map $\theta \colon (A, \mathcal{B}\cap A,\mu|_A) \to (X,\mathcal{B},\mu)$ satisfying $\theta(x) \in X^s(x,\eps_X)$ for all $x\in A$. Though the definition of the map $\theta$ itself is not complicated (see Lemma \ref{lem: second half}), it is difficult to directly establish the $C\ep$-measure-preserving property. 

Therefore, we construct a new Borel measure $\lambda_0$ for which $\theta$ is \textit{measure-preserving} as a map from $(A,\mathcal{B}\cap A,\lambda_0|_A)$ and $(X,\mathcal{B},\lambda_0)$.

We then make use of Proposition \ref{prop: perturbation of measure preserving} to prove that $\theta$ is $C\ep$-measure-preserving property as a map from $(A,\mathcal{B}\cap A,\mu|_A)$ to $(X,\mathcal{B},\mu)$.

We cannot guarantee this property for every $E\in\mathscr{B}_1$. Hence, we will further exclude some partition elements from $\mathscr{B}_1$, and in doing so, obtain our desired collection $\mathscr{B}$. For now, we will assume that we work with a given $E\in\mathscr{B}_1$.

Recalling that on each $(R_i,\mathcal{B}\cap R_i)$ there is a Borel probability measure $\lambda_i$ given by $\lambda_i := \rho_i\times \nu_i$ we define a new probability measure
\begin{equation}\label{eq: lambda_0}
\lambda_0 := \left(\sum_{i=1}^n \mu(R_i)\right)^{-1}\left( \sum_{i=1}^n\mu(R_i)\lambda_i\right)
\end{equation}
obtained as the weighted sum of the measures $\{\lambda_i\}_{i=1}^n$ normalized to a probability measure on $(\bigsqcup_{i=1}^n R_i,\mathcal{B}\cap \bigsqcup_{i=1}^nR_i)$. For convenience, for all $i$, define $\lambda_i(Z) = 0$ for all $Z\in\mathcal{B}$ for which $Z$ is not already $\lambda_i$-measurable. This ensures that $\lambda_0$ is a measure with respect to $(X,\mathcal{B})$ while keeping our intuition that each $\lambda_i$ can be treated as just a measure on $(R_i,\mathcal{B}\cap R_i)$.

We now will verify that Proposition \ref{prop: perturbation of measure preserving} can be applied with:

\begin{align*}
(X_1,\mathcal{A}_1,\mu_1) &= (A,\mathcal{B}\cap A,\mu|_A), & (X_2,\mathcal{A}_2,\mu_2) &= (X,\mathcal{B},\mu),\\
(Y_1,\nu_1) &= (A\cap G, \lambda_0|_A), & (Y_2,\nu_2) &= (G,\lambda_0).\\
\end{align*}

\begin{rem}\label{rem: applying 4.10}
With the above choice of $X_i,Y_i,\nu_i,\mu_i$'s, we have already satisfied some of the necessary conditions to apply Proposition \ref{prop: perturbation of measure preserving}.
In particular,
$$\mu(A\cap G)\geq (1-\ep)\mu(A)  \text{ and } \mu(G) \geq 1-\ep^2,$$
by Lemma \ref{lem: A subset E} and Corollary \ref{cor: general covering}. These verify the conditions that $\mu_1(Y_1) \geq 1-\ep$ and $\mu_2(Y_2) \geq 1-\ep$ necessary to apply Proposition \ref{prop: perturbation of measure preserving}. 

\end{rem}

We now verify the second condition of Proposition \ref{prop: perturbation of measure preserving} on $\frac{d\mu}{d\lambda_0}$.
\begin{lem}\label{lem: B and G}
	For any measurable $B\subseteq X$, we have
	$$(1-2\ep^2)\lambda_0(B \cap G) \leq \mu(B \cap G) 
	\leq (1+\ep^2)\lambda_0(B \cap G).$$
In particular, $\abs{\frac{d\mu}{d\lambda_0} - 1 } < 2\ep^2$ on $G$.
\end{lem}
\begin{proof}
	Note that, because each $G_i \subset R_i$, we have
	$$\mu(B \cap G) = \sum\limits_{i=1}^n \mu(B \cap G_i) = \sum\limits_{i=1}^n \mu(R_i)\cdot  \mu |_{R_i}(B \cap G_i).$$ By our choice of rectangles, $\abs{\frac{d\mu|_{R_i}}{d\lambda_i}\Big|_{G_i}  - 1} < \ep^2$, and so it follows that
	$$(1-\ep^2)\left(\sum_{i=1}^n \mu(R_i)\right)\lambda_0(B \cap G) \leq \mu(B \cap G) 
	\leq (1+\ep^2)\left(\sum_{i=1}^n \mu(R_i)\right)\lambda_0(B \cap G).$$
	Since
	$$1 - \ep^2 \leq \mu(G) \leq \sum\limits_{i=1}^n \mu(R_i) \leq 1,$$ we are done.
\end{proof}

The following corollary will be used in verifying the remaining condition of Proposition \ref{prop: perturbation of measure preserving}.

\begin{cor}\label{cor: mu and lambda0}
For any measurable $B \subseteq G$, we have
$$\abs{\frac{d\mu|_{B}}{d\lambda_0|_{B}} - 1 } < 2\ep.$$
\end{cor}
\begin{proof}
This is a direct application of Proposition \ref{prop: rescaling derivative} using Lemma \ref{lem: B and G}.
\end{proof}

The final condition to be verified in order to apply Proposition \ref{prop: perturbation of measure preserving} is
\begin{equation}\label{eq: last to show}
\abs{\frac{d\mu|_{A}}{d\lambda_0|_{A}}\Big|_{A \cap G} - 1} < \ep
\end{equation}
and the existence of a  $(\lambda_0|_A,\lambda_0)$-measure-preserving map $\theta \colon (A, \mathcal{B}\cap A, \lambda_0|_A) \to (X,\mathcal{B},\lambda_0|_G)$, which is injective and bimeasurable. We establish this last condition in the next subsection.

\subsection{Verifying the final condition and constructing \texorpdfstring{$\theta$}{theta}}

Unlike the other conditions of Proposition \ref{prop: perturbation of measure preserving} that were relatively simple to verify, the last remaining condition is considerably more difficult to verify. 
This is because we cannot directly compare $\lambda_0(A)$ to $\lambda_0(E)$ or $\lambda_0(A \cap G)$ to $\lambda_0(E \cap G)$ as we did for $\mu$. So, we begin by observing that for all positive measure $B\subset A\cap G$,
\begin{align*}
\frac{\mu|_A(B)}{\lambda_0|_A(B)} &= \frac{\mu(B)}{\mu(A)}\frac{\lambda_0(A)}{\lambda_0(B)}
\\
&= \frac{\mu(B)}{\mu(A\cap G)}\frac{\mu(A\cap G)}{\mu(A)}\frac{\lambda_0(A)}{\lambda_0(A\cap G)}\frac{\lambda_0(A\cap G)}{\lambda_0(B)}
\\
&= \left(\frac{\mu(A\cap G)}{\mu(A)}\right)\left(\frac{\mu|_{A\cap G}(B)}{\lambda_0|_{A\cap G}(B)}\right)\left(\frac{\lambda_0(A)}{\lambda_0(A\cap G)}\right).
\end{align*}

From Lemma \ref{lem: A subset E} and Corollary \ref{cor: mu and lambda0}, we already have bounds on the first two terms of this expression, namely that the first term is at least $1-\ep$ and that the second term is between $1-2\ep$ and $1+2\ep$.
Thus, we devote the next two lemmas to establishing an upper bound on $\displaystyle \frac{\lambda_0(A)}{\lambda_0(A\cap G)}$. To get an effective bound, we further restrict (from $\mathscr{B}_1$) the elements of $\bigvee_{j=m_1}^{m_2}f^j\xi$ with which we work.

\begin{lem}\label{lem: B2}
There exists $\mathscr{B}_2\subset \mathscr{B}_1$ with $\mu\left(\bigcup\limits_{E\in\mathscr{B}_2}E\right) \geq 1-\ep$  such that for all $E \in \mathscr{B}_2$,
$$\lambda_0(E\cap G) \geq (1-\ep)\lambda_0(E).$$
\end{lem}
\begin{proof}
Using the fact that $\lambda_i(G) > 1 - \ep^2$, it follows that $\lambda_0(G^c) < \ep^2$ by the following calculation:
\begin{align*}
	\lambda_0(G^c) &= \left(\sum_{i=1}^n \mu(R_i)\right)^{-1} \sum_{i=1}^n \mu(R_i)\lambda_i(R_i\setminus G_i),
	\\
	&\leq \left(\sum_{i=1}^n \mu(R_i)\right)^{-1} \sum_{i=1}^n \mu(R_i)\ep^2= \ep^2.
\end{align*}
Therefore, applying Lemma \ref{lem: large intersections}, we have that the $\lambda_0$-measure of the union of all partition elements $E \in \bigvee_{j=m_1}^{m_2} f^j\xi$ with $\lambda_0(E\cap G) < (1-\ep)\lambda_0(E)$ is at most $\ep$.
Among such partition elements, we denote by $\mathscr{D}$ the set of all $E$ which also belong to $\mathscr{B}_1$. Then $\lambda_0\left(\bigcup\limits_{E\in\mathscr{D}} E\right) < \ep$.  
Since $\mu\left(\bigcup\limits_{E\in\mathscr{B}_1}E\right) \geq 1-\ep/2$ from Lemma \ref{lem: A subset E}, by setting $\mathscr{B}_2:=\mathscr{B}_1 \setminus \mathscr{D}$, showing that $\mu(\mathscr{D}) <\ep/2$ will complete our proof.

Since $\mathscr{B}_2 \subset \mathscr{B}_1$, each $E \in \mathscr{B}_2$ satisfies $\mu(E \cap G^c) \leq \ep\mu(E)$.
Then we have
\begin{align*}
	\mu\left(\bigcup_{E\in\mathscr{D}} E\right) &= \sum_{E\in\mathscr{D}} \Big(\mu(E\cap G) + \mu(E\cap G^c)\Big)
	\\
	&\leq \sum_{E\in\mathscr{D}} \Big((1+\ep^2)\lambda_0(E\cap G) + \ep \mu(E)\Big)
	\\
	&\leq (1+\ep^2)\lambda_0(\bigcup_{E\in\mathscr{D}} E) + \ep\mu(\bigcup_{E\in\mathscr{D}} E)
	\\
	&\leq (1+\ep^2)\ep + \ep\mu(\bigcup_{E\in\mathscr{D}} E),
\end{align*}
where the first inequality uses Lemma \ref{lem: B and G}. Consequently,
$\displaystyle \mu\left(\bigcup_{E\in\mathscr{D}}E\right) < \frac{1+\ep^2}{1-\ep}\ep < 2\ep$.
\end{proof}

Using Lemma \ref{lem: B2}, we establish the last remaining condition needed to apply Proposition \ref{prop: perturbation of measure preserving} (see Remark \ref{rem: applying 4.10}).

\begin{lem}
For every $E\in \mathscr{B}_2$ the subset $A\subseteq E$ from Lemma \ref{lem: A subset E} satisfies
$$\lambda_0(A\cap G)\geq (1-\ep)\lambda_0(A).$$
Moreover, for such $A$ we have $$\displaystyle \abs{\frac{d\mu|_{A}}{d\lambda_0|_{A}}\Big|_{A \cap G} - 1} < 6\ep.$$
\end{lem}
\begin{proof}
Since $$E \cap G = \big(A \cup (E \setminus A)\big) \cap G = \big(A \cap G\big) \cup \big((E\setminus A) \cap G\big)$$ and we have an estimate on $\lambda_0(E\cap G)$ from Lemma \ref{lem: B2} for every $E\in \mathscr{B}_2$, we will estimate the $\lambda_0$-measure of $(E\setminus A) \cap G$ below.

For every $E\in \mathscr{B}_2 \subseteq \mathscr{B}_1$, Lemma \ref{lem: B and G} and \eqref{eq: E n G} give 
$$\mu(E) \leq \frac{1}{1-\ep}\mu(E\cap G) \leq \frac{1+\ep^2}{1-\ep}\lambda_0(E\cap G) \leq (1+2\ep)\lambda_0(E).$$
In the reverse direction, \eqref{eq: A and E} implies that
$$\mu(E\cap A) \geq \left(1-\frac{\ep}{2}\right)\mu(E)$$
and so using Lemma \ref{lem: B and G} we have
$$\frac{\ep}{2} \mu(E) \geq \mu(E\cap A^c) \geq \mu((E\setminus A) \cap G) \geq (1-2\ep^2)\lambda_0((E\setminus A) \cap G).$$

Combining the two gives $$\displaystyle \lambda_0((E\setminus A) \cap G) \leq \frac{\ep(1+2\ep)}{2(1-2\ep^2)} \lambda_0(E) \leq \ep \lambda_0(E).$$

Now the first statement of the lemma follows:
\begin{align*}
	\lambda_0(A\cap G) &= \lambda_0(E\cap G) - \lambda_0((E\setminus A) \cap G)
	\\
	&\geq (1-\ep)\lambda_0(E) - \ep\lambda_0(E)
	\\
	&\geq (1 - 2\ep)\lambda_0(A)
\end{align*}
where we have used Lemma \ref{lem: B2} in the first inequality.

Since $\displaystyle \frac{1}{1-2\ep} \leq 1+3\ep$, it follows that $\displaystyle \frac{\lambda_0(A)}{\lambda_0(A\cap G)} \in (1, 1+3\ep)$. This observation together with the discussion preceding Lemma \ref{lem: B2} establishes the second statement. In particular, for any measurable $B\subset A\cap G$ we have
\begin{align*}
\frac{\mu|_A(B)}{\lambda_0|_A(B)} &= \left(\frac{\mu(A\cap G)}{\mu(A)}\right)\left(\frac{\mu|_{A\cap G}(B)}{\lambda_0|_{A\cap G}(B)}\right)\left(\frac{\lambda_0(A)}{\lambda_0(A\cap G)}\right)
\\
&\leq (1+2\ep)(1+3\ep)
\\
&\leq 1+6\ep,
\end{align*}
while the lower bound is $1-3\ep$.
\end{proof}

The next lemma constructs the required map $\theta \colon A \to X$ and applies Proposition \ref{prop: perturbation of measure preserving}. Since $\mu\left(\bigcup\limits_{E\in\mathscr{B}_2}E\right) \geq 1-\ep$, this concludes the proof of Proposition \ref{prop: pre most E}.

\begin{lem}\label{lem: second half} For every $E\in \mathscr{B}_2$ and its corresponding subset $A \subset E$ from Lemma \ref{lem: A subset E}, there exists an injective, bimeasurable, measure-preserving map $\theta : (A,\mathcal{B}\cap A,\lambda_0|_A)\to (X,\mathcal{B},\lambda_0)$ where $\theta(x) \in X^s(x,\eps_X)$ for all $x\in A$.
\end{lem}
\begin{proof}

Recall that each rectangle $R_i$ can be identified with the product $Y_i\x Z_i$, where $Y_i = X^s(z,\eps_X)\cap R_i$ and $Z_i = X^u(z,\eps_X)\cap R_i$ for some $z\in R_i$. Then $R_i = [Y_i,Z_i]$ and is equipped with the product measure $\rho_i \times \nu_i$.
We begin by defining $\theta : A\cap R_i \to R_i$. As $\rho_i$ is non-atomic and both $(A\cap Y_i,\mathcal{B}\cap A\cap Y_i,\rho_i|_{A\cap Y_i})$ and $(Y_i,\mathcal{B}\cap Y_i,\rho_i)$ are Lebesgue spaces (by completeness of $\rho_i$), there is a measure-preserving isomorphism 
$$\theta_i\colon(A\cap Y_i),\mathcal{B}\cap A\cap Y_i,\rho_i|_A)\to (Y_i,\mathcal{B}\cap Y_i,\rho_i).$$
Because $A$ intersects each $R_i$ layerwise by construction, this mapping extends to $A \cap  R_i$ along the unstable direction. In particular, for $y\in A\cap R$, we can define
$$\theta : (A\cap R_i,\mathcal{B}\cap A\cap R_i, \lambda_i|_A)\to (R_i,\mathcal{B}\cap R_i, \lambda_i)$$
by
$$\theta(y) := [y,\theta_i([z,y])].$$
Then $\theta$ is still measure-preserving and satisfies the property that $\theta(x) \in X^s(x,\eps_X)$ for all $x\in R_i$. Furthermore, for all measurable $B\subset A$, $\theta(B)$ is measurable. This is because the bracket operation is a homeomorphism, and $\theta_i$ is an isomorphism.

Moreover, extending the map to all of our rectangles, we have that $\theta : (A\cap \bigsqcup R_i, \lambda_0|_A)\to (\bigsqcup R_i, \lambda_0)$ is measure-preserving since $\lambda_i|_A$ is a scalar multiple of $\lambda_0|_A$ on each $R_i$. Therefore, 
$$\theta : (A,\mathcal{B}\cap A,\lambda_0|_A) \to (X,\mathcal{B},\lambda_0)$$ is measure-preserving as well, because $\lambda_0(X\setminus \bigsqcup R_i) = 0$ from its construction \eqref{eq: lambda_0}. 

\end{proof}

All of the results in this section prove our desired result using Proposition \ref{prop: perturbation of measure preserving}:

\begin{cor}
For every $E\in\mathscr{B}_2$ and its corresponding subset $A\subset E$ from Lemma \ref{lem: A subset E}, there exists a $30\eps$-measure-preserving map
$$\theta : (A,\mathcal{B}\cap A,\mu|_A)\to (X,\mathcal{B},\mu)$$
with $\theta(x)\in X^s(x,\eps_X)$ for all $x\in A$.
\end{cor}

\bibliographystyle{amsalpha}
\bibliography{bernoulli}
\end{document}